\documentclass[a4paper]{paper}
\usepackage{amsmath,amsthm,amsfonts,mathrsfs}
\usepackage{amssymb}
\usepackage{tikz-cd}
\usepackage{xspace}
\usepackage{graphicx}
\usepackage{paralist}
\usepackage{subfig}
\usepackage{hyperref}
\hypersetup{%
    pdfmenubar=true,       
    pdffitwindow=false,     
    pdfstartview={FitH},    
    pdftitle={Solution theory of discrete basis sinograms in multispectral CT},
    colorlinks=true,       
    linkcolor=red,          
    citecolor=red,        
    filecolor=magenta,      
    urlcolor=cyan,           
}
\usepackage{acro}
\usepackage{cleveref}
\usepackage{algorithm}
\usepackage[noend]{algpseudocode}
\makeatletter
\def\BState{\State\hskip-\ALG@thistlm}
\makeatother

\Crefname{equation}{}{}
\crefname{equation}{}{}
\usepackage{todonotes}
\usepackage{inputenc}
\usepackage{dsfont}

\theoremstyle{plain}
 \newtheorem{theorem}{Theorem}[section]
 \newtheorem{proposition}{Proposition} 
 \newtheorem{lemma}{Lemma}[section]
 \newtheorem{corollary}{Corollary}[section]
 
 \newtheorem{assumption}{Assumption}
 
\theoremstyle{definition}
 \newtheorem{definition}{Definition}

\theoremstyle{remark}
  \newtheorem{remark}{Remark}

\newcommand{\eg}{\textsl{e.g.}}
\newcommand{\ie}{\textrm{i.e.}}

\newcommand{\Real}{\mathbb{R}}
\newcommand{\domain}{\Omega} 
\newcommand{\Diff}{\mathcal{D}} 
\DeclareMathOperator*{\diag}{\texttt{diag}}
\newcommand{\bx}{\boldsymbol{x}}
\newcommand{\by}{\boldsymbol{y}}

\newcommand{\bdf}{\boldsymbol{f}}
\newcommand{\bdg}{\boldsymbol{g}}

\newcommand{\bdb}{\boldsymbol{b}}

\newcommand{\bdn}{\boldsymbol{n}}
\newcommand{\bds}{\boldsymbol{s}}
\newcommand{\bzero}{\boldsymbol{0}}
\newcommand{\tra}{\boldsymbol{\mathsf{T}}}

\newcommand{\bF}{\boldsymbol{F}}
\newcommand{\bH}{\boldsymbol{H}}
\newcommand{\bzeta}{\boldsymbol{\zeta}}

\begin{document}

\title{Solution existence, uniqueness, and stability of discrete basis sinograms in multispectral CT}
\author{
 Yu Gao\thanks{LSEC, ICMSEC, Academy of Mathematics and Systems Science, Chinese Academy of Sciences, Beijing 100190, China. University of Chinese Academy of Sciences, Beijing 100049, China.}
 \and
 Xiaochuan Pan\thanks{Department of Radiology, The University of Chicago, Chicago, IL 60637, USA.}
 \and
 Chong Chen\thanks{LSEC, ICMSEC, Academy of Mathematics and Systems Science, Chinese Academy of Sciences, Beijing 100190, China. University of Chinese Academy of Sciences, Beijing 100049, China.}
}

\maketitle

\begin{abstract}

This work investigates conditions for quantitative image reconstruction in multispectral computed tomography (MSCT), which remains a topic of active research. In MSCT, one seeks to obtain from data the spatial distribution of linear attenuation coefficient, referred to as a virtual monochromatic image (VMI), at a given X-ray energy, within the subject  imaged. As a VMI is decomposed often into a linear combination of basis images with known decomposition coefficients, the reconstruction of a VMI is thus tantamount to that of the basis images. An empirical, but highly effective, two-step data-domain-decomposition (DDD) method has been developed and used widely for quantitative image reconstruction in MSCT. In the two-step DDD method, step (1) estimates the so-called basis sinogram from data through solving a nonlinear transform, whereas step (2) reconstructs basis images from their basis sinograms estimated. Subsequently, a VMI can readily be obtained from the linear combination of basis images reconstructed. As step (2) involves the inversion of a straightforward linear system, step (1) is the key component of the DDD method  in which a nonlinear system needs to be inverted for estimating the basis sinograms from data. In this work, we consider a {\it discrete} form of the nonlinear system in step (1), and then carry out theoretical and numerical analyses of conditions on the existence, uniqueness, and stability of a solution to the discrete nonlinear system for accurately estimating the discrete basis sinograms, leading to quantitative reconstruction of VMIs in MSCT.
\end{abstract}

\section{Introduction}\label{sec:Introduction} 

The advancement of hardware and application in multispectral computed tomography (MSCT) has prompted an increased level of research interest in investigating image reconstruction in MSCT.  For a given X-ray energy, one seeks to determine in MSCT quantitatively the spatial distribution of linear attenuation coefficient (LAC), which is referred to also as a virtual monochromatic image (VMI), within the subject scanned. Because the VMI is decomposed often into a linear combination of basis images with known decomposition coefficients, the reconstruction of a VMI is tantamount to that of the basis images. Development of algorithms, including one-step and two-step algorithms, for reconstruction of basis images in MSCT constitutes an important topic of active research in the CT field. While one-step algorithms have been investigated in recent years for reconstructing basis images directly from data \cite{zhaozz14,LongFessler14,BarberPan16,chpan17,chpan21,gao2021EPD}, there remains of theoretical and practical interest in study of the two-step data-domain-decomposition (DDD) method because it is used widely for reconstruction of basis images in dual-energy CT (DECT), a special form of MSCT, and because the study may yield useful insights into the development of one-step algorithms.

In the two-step DDD method, step (1) estimates the so-called basis sinogram from data through solving a nonlinear transform, whereas step (2) reconstructs basis images from their basis sinograms estimated. Subsequently, a VMI can readily be obtained as a linear combination of basis images reconstructed. As step (2) involves the inversion of a straightforward linear system, step (1) constitutes the key component of the DDD method in which a nonlinear system needs to be inverted for estimating the basis sinograms from data. Therefore, it is theoretically and practically worthy to study the conditions on existence, uniqueness, and stability of the solution to the nonlinear system.

There exist works on investigating the conditions on the existence, uniqueness, and stability of the inversion of the nonlinear system in step (1) in a {\it continuous} form (see \cref{eq:cc_formula} below), which is referred to simply as the continuous nonlinear system hereinafter. Alvarez studied the invertibility in terms of the zero Jacobian determinant of the continuous nonlinear system in DECT \cite{Alvarez19Invert}.  For MSCT, Bal and Terzioglu \cite{Bal_2020} have performed recently an interesting analysis of the conditions on the existence, uniqueness, and stability of the solution to the continuous nonlinear system in \cref{eq:cc_formula}, and they also carried out numerical studies verifying that the mapping (i.e., the continuous nonlinear system) is injective in some closed rectangle. Ding et al.  provided a sufficient condition for the invertibility of the multi-energy X-ray transform also from the continuous viewpoint \cite{Ding21_invert}.

In practical MSCT, however, the nonlinear and linear systems involved must be in {\it discrete} forms as the basis sinograms, basis images, and data are in discrete forms (see \cref{eq:linear_part} and \cref{eq:nonlinear_systems_original} below), which are referred to, respectively, simply as discrete linear and nonlinear systems hereinafter. In this work, we carry out theoretical and numerical analyses of conditions on the existence, uniqueness, and stability of a solution to the {discrete} nonlinear system  (see \cref{eq:nonlinear_systems_original} below) for accurately estimating basis sinograms in MSCT, as there appears to be a lack of such analyses reported in literature.

The paper is organized as follows. Following the introduction above, we describe the data models in continuous and discrete forms in \cref{sec:Data-models} and mathematical preliminaries in \cref{sec:Preliminaries}. We then perform in \cref{sec:math_analy} theoretical analyses of the conditions on the existence, uniqueness, and stability of the solution to the discrete nonlinear system in \eqref{eq:nonlinear_systems_map}, followed by numerical studies in \cref{sec:Numerical} demonstrating the theoretical results. Finally, discussion and remarks are made in \cref{sec:discussion_conclusion}.

\section{Data models in multispectral CT}\label{sec:Data-models} 
\subsection{Continuous-to-continuous (CC)-data model}
In MSCT, data are measured from an object scanned for multiple spectra $s^{[q]}(E)$, where $q=1, 2, \ldots, Q$, and $Q$ denotes the total number of distinct spectra involved. 
For a given spectrum{\footnote{The effective spectrum of a CT system is indeed the product of the X-ray-tube spectrum and detector-energy response.}}, data can be measured along X-rays each of which is specified by $l(\hat{\theta}^{[q]})$, where $\hat{\theta}^{[q]}$ denotes the ray direction. Measured data in the absence of other physical factors can be modeled as 
\begin{equation}\label{eq:multimeasurement}
g_{l(\hat{\theta}^{[q]})} = \ln\int_{0}^{E_{\text{max}}} s^{[q]}(E)\exp\left(-\int_{l (\hat{\theta}^{[q]})} \mu(\by, E)\mathrm{d}l\right) \mathrm{d} E,
\end{equation}
where spectrum $s^{[q]}(E)$ is normalized over energy satisfying $\int_{0}^{E_{\text{max}}}s^{[q]}(E) \mathrm{d} E = 1$,
$\mu(\by, E)$ denotes the energy-dependent LAC 
of interest at spatial position $\by \in \domain \subset \mathbb{R}^d$ ($d=2$ or $3$), energy $0<E<E_{\text{max}}$, and $E_{\text{max}}$ the maximum energy in the scan. 
LAC $\mu(\by, E)$ is decomposed often as a linear combination 
\begin{equation}\label{eq:attenuationco}
\mu(\by, E) = \sum_{k=1}^K b_k(E) f_k(\by), 
\end{equation}
of basis image $f_k(\by)$ that is a function only of $\by$, where $b_k(E)$ denotes expansion coefficient, $k=1, 2, ..., K$, and $K$ the number of basis images. For example, for the typical X-ray-energy range in diagnostic DECT, $K=2$ is chosen often as photoelectric effect and Compton scatter are the dominant factors contributing to $\mu(\by, E)$; $b_k(E)$ denotes the known mass attenuation coefficient (MAC) of the $k$-th basis material (\eg, water and bone); and $f_k(\by)$ represents the material-equivalent density function (or basis image) of the $k$-th basis material \cite{alma76,hsieh15}. 

Substituting \cref{eq:attenuationco} 
into \cref{eq:multimeasurement}, we obtain 
\begin{equation}\label{eq:cc_formula}
g_{l(\hat{\theta}^{[q]})}= \ln\int_{0}^{E_{\text{max}}} s^{[q]}(E)\exp\left(-\sum_{k=1}^K b_k(E) x_{k}^{l(\hat{\theta}^{[q]})} \right)\mathrm{d} E,
\end{equation}
where 
\begin{equation}\label{eq:sinoCC_formula}
x_k^{l(\hat{\theta}^{[q]})}= \int_{l(\hat{\theta}^{[q]})}f_k(\by)\mathrm{d} l.
\end{equation}
We refer to \cref{eq:cc_formula} as the {\it continuous-to-continuous} (CC)-data model 
\cite{chpan17} and to $x_k^{l(\hat{\theta}^{[q]})}$ in \cref{eq:sinoCC_formula} as the continuous basis sinogram of continuous basis image $f_k(\by)$,  because $\hat{\theta}^{[q]}$, $E$, and $\by$ are continuous variables. 

Using knowledge of $b_k(E)$  and $f_k(\by)$ in \cref{eq:attenuationco}, $\mu(\by, E)$, also referred to as the VMI at given energy $E$, can readily be obtained. Therefore, in MSCT,  the task of reconstructing VMI $\mu(\by, E)$ is tantamount to the task of reconstructing basis images $ f_k(\by)$ for $k=1, ..., K$. 

For MSCT with a geometrically-consistent scan configuration, data are collected using identical scan geometry for all $Q$ spectra, i.e., $\hat{\theta}^{[q]}=\hat{\theta}$ is independent of $q$.

The two-step DDD method has been devised for image reconstruction in geometrically-consistent MSCT in which step (1) estimates $x_k^{l(\hat{\theta}^{[q]})}$ from knowledge of $g_{l(\hat{\theta}^{[q]})}$ through solving continuous nonlinear system \cref{eq:cc_formula}, whereas step (2) reconstructs $f_k(\by)$ from knowledge of $x_k^{l(\hat{\theta}^{[q]})}$ estimated through solving continuous linear system \cref{eq:attenuationco}. As step (2) involves solving a continuous linear system, step (1) that solves the continuous nonlinear system constitutes the key component of the DDD method. Works \cite{Bal_2020,Bal_2022,Ding21_invert} have been reported on investigating the conditions on the existence, uniqueness, and stability of a solution to the continuous nonlinear system in \cref{eq:cc_formula}.

\subsection{Discrete-to-discrete (DD)-data model}
\label{sec:dd_model}
In practical imaging, data can be acquired only for discrete rays and energies, and an image is reconstructed generally on an array of discrete voxels. Considering data and image arrays in practical imaging, one can devise a discrete-to-discrete (DD)-data model \cite{chpan17}, which is created often by applying a specific discrete scheme selected to the CC-data model in \cref{eq:cc_formula}. In particular, taking a discrete scheme in which energy range $[0, E_{\text{max}}]$ is divided into $M$ intervals of equal size $\Delta_E$, whereas the image space is represented by an array of size $I=I_x\times I_y \times I_z$, consisting of identical cubic voxels of size $\Delta_x=\Delta_y=\Delta_z$. Let $s^{[q]}_m=s(m \,\Delta_E)$ and $b_{km}=b_k(m\, \Delta_E)$, where $m=1, 2, \ldots, M$. 

We use vector $\bdf_k$ of size $I$ to denote discrete basis image $k$ in a concatenated form in the order of $x$, $y$, and $z$, with entry $f_{ki}$ depicting the value of discrete basis image $k$ at voxel $i$, where $i=1, 2, ..., I$. We also use vector ${\boldsymbol{\mu}}_m$ of size $I$ to denote the VMI at energy $m$ in a concatenated form in the order of $x$, $y$, and $z$, with entry $\mu_{m i}$ depicting the value of discrete VMI $m$ at voxel $i$. Observing  \cref{eq:attenuationco}, we can obtain a relationship between discrete VMI ${\boldsymbol{\mu}}_m$ and discrete basis images $\bdf_k$ as
\begin{equation}\label{eq:attenuationco-DD}
\mu_{m i} = \sum_{k=1}^K b_{km} f_{ki}.
\end{equation}

We assume that for given spectrum $s^{[q]}(E)$, data are acquired for a total of $J_q$ rays specified by a slew of discrete orientations $\hat{\theta}_j^{[q]}$, where $j=1, 2, \ldots, J_q$, and $g_j^{[q]}$ denotes the measurement for ray $j$. Subsequently, we can obtain a {DD-data} model as
\begin{equation}\label{eq:dd_formula}
g_j^{[q]}  = \ln\sum_{m=1}^{M} s^{[q]}_{m} \exp\left(-\sum_{k=1}^K b_{km} x^{[q]}_{j k}
\right) \quad \text{for $1 \le j \le J_q$,}
\end{equation}
where $\sum_{m=1}^{M} s^{[q]}_{m} \Delta_E =1$, and 
\begin{equation}\label{eq:linear_part}
x^{[q]}_{j k} = \sum_{i=1}^I a^{[q]}_{ji} f_{ki}
\end{equation}
is referred to as the discrete basis sinogram of discrete basis image $k$, and $a^{[q]}_{ji}$ denotes a weight applied to value $f_{ki}$ of discrete basis image $k$ at voxel $i$, and it is selected often as the intersection length of ray $j$ with voxel $i$ in the 
image array \cite{chen2020fast}. It is well-known \cite{EDEC_06,Zou08,JDE_MBIR14} that from 
knowledge of discrete basis sinograms $\{x^{[q]}_{jk}\}$, one can solve the discrete linear system in \eqref{eq:linear_part} to 
obtain discrete basis images $\bdf_k$ from which discrete VMI $\boldsymbol{\mu}_m$ can readily be obtained by use of \cref{eq:attenuationco-DD}.

In applications of diagnostic MSCT and DECT, geometrically-consistent scan configurations are used, and data are acquired with all $Q$  different spectra $s^{[q]}(E)$ for each of all of the rays, i.e., $a_{ji}^{[q]}$ and $J_q$ are identical for all spectra and thus are independent of spectra. We can then use $a_{ji}$ and $J$ to replace $a_{ji}^{[q]}$ and $J_q$ as they are independent of $q$, and rewrite the DD-data model as
\begin{equation}\label{eq:nonlinear_systems_original}
g_j^{[q]}  = \ln\sum_{m=1}^{M} s^{[q]}_{m} \exp\left(-\sum_{k=1}^K b_{km} x_{j k}\right) \quad \text{for $1 \le j \le J$.}
\end{equation}
%

For the $j$-th X-ray, using vectors $\bx_j$ and $\bdg_j$ to denote the corresponding basis 
sinogram and data along this ray, where 
$\bx_j = [x_{j1}, \ldots, x_{jK}]^{\tra}$ and $\bdg_j=[g^{[1]}_j,\ldots,g^{[Q]}_j]^{\tra}$.
We re-express \eqref{eq:nonlinear_systems_original} as
\begin{equation}\label{eq:nonlinear_systems_map}
\bF(\bx_j) = \bdg_j \quad \text{for $1 \le j \le J$,}
\end{equation}
where nonlinear mapping $\bF : \Real^K \rightarrow \Real^Q$ is given as
$\bF(\bx) = [F_1(\bx), \ldots, F_Q(\bx)]^{\tra}$, $\bx = [x_{1}, \ldots, x_{K}]^{\tra}$, and
\begin{align*}
    F_q(\bx) := \ln\sum_{m=1}^{M} s_{m}^{[q]} \exp\left(-\sum_{k=1}^K b_{km}x_{k}\right).
\end{align*}

\subsection{Existence, uniqueness and stability of solution to \cref{eq:nonlinear_systems_map}}

In the discrete two-step DDD method for MSCT with geometrically consistent scan configurations, step (1) estimates discrete basis sinogram $\mathbf{x}=[\bx_1, \ldots,\bx_J]^{\tra}$ from discrete data $\mathbf{g}= [\bdg_1, \ldots, \bdg_J]^{\tra}$ through solving the discrete nonlinear system in \cref{eq:nonlinear_systems_original} or, equivalently, in \cref{eq:nonlinear_systems_map}, and step 2 reconstructs discrete basis images  $\bdf_k$ from estimated $\mathbf{x}$ by solving discrete linear system \cref{eq:linear_part}. The challenge to the discrete two-step method also lies in step (1) as it solves a discrete nonlinear system, while step (2) solves a straightforward discrete linear system. 
In this work, we thus focus on investigating the conditions on the existence, uniqueness, and stability of solution $\bx$ to 
the discrete nonlinear system in \cref{eq:nonlinear_systems_map}. 

Mathematically, the existence means that for given measurement $\bdg$, there exists $\bx^{\ast}$ 
such that $\bF(\bx^{\ast}) = \bdg$, which corresponds to the surjection, whereas the uniqueness is equivalent to the injectivity of mapping $\bF$, 
namely, $\bF(\bx_1)=\bF(\bx_2)$ implies $\bx_1=\bx_2$. The solution stability, on the other hand, is characterized by $\| \bx_1 - \bx_2\| \le \gamma \| \bF(\bx_1) - \bF(\bx_2)\|$, where $\bF(\bx_1)$ and $\bF(\bx_2)$ denote two measurements, and $\gamma$ is a positive constant \cite{bal2012introduction}.

\section{Mathematical preliminaries}
\label{sec:Preliminaries}

We introduce mathematical preliminaries required.  Note that the nonlinear mapping 
in \cref{eq:nonlinear_systems_map} is defined on a finite-dimensional Euclidean space, which can be equipped with inner product $\langle \cdot,\cdot\rangle$ and Euclidean norm $\|\cdot\| = \sqrt{\langle \cdot,\cdot\rangle}.$ Let $\Diff \bH(\bx)$ denote the Jacobian matrix of mapping $\bH$ at $\bx$, $\det(A)$ the determinant of square matrix $A$, and $\diag(\bx)$ the diagonal matrix with diagonal $\bx$, respectively.

We define index set $\langle M \rangle: = \{1,2,\ldots,M\}$ and complementary set $\alpha^c := \langle M \rangle \setminus  \alpha$, where $\alpha$ is a subset of $\langle M \rangle$. We also use $\#\alpha$ to denote the number of elements in set $\alpha$.  For matrix $A$, which can be a non-square matrix, we use $A[\alpha,\beta]$ to denote the submatrix with rows and columns indexed by index sets $\alpha$ and $\beta$, respectively, and abbreviate $A[\alpha,\alpha]$ as $A[\alpha]$. In particular, for square matrix $A$, we refer to $A[\alpha]$ as a principal submatrix of $A$. The submatrix's determinant,  $\det(A[\alpha,\beta])$ with $\#\alpha = \#\beta$, is referred as a minor of $A$, whereas  $\det(A[\alpha])$ is referred as a principal minor of $A$. 


\begin{definition}\cite{johnson_smith_2020}\label{def:P-matrix}
    If $A$ is a square matrix, $A$ is referred to as a P-matrix if every principal minor of $A$ is positive, or as a weak P-matrix if $\det(A)>0$ and if every other principal minor is non-negative.
\end{definition}

\begin{definition}\cite{Parthasarathy1983}\label{def:homeo}
Mapping $\bH: \mathcal{X} \rightarrow \mathcal{Y}$ between two topological spaces is a homeomorphism if it has the following properties: 
\begin{itemize}
    \item[(i)] $\bH$ is a bijection, namely injection and surjection (one-to-one and onto),
    \item[(ii)] $\bH$ its inverse function $\bH^{-1}$ are continuous.
\end{itemize}
\end{definition}

\begin{definition}\cite{Parthasarathy1983}\label{def:local_homeo}
Mapping $\bH : \mathcal{X} \rightarrow \mathcal{Y}$ is referred to as a local homeomorphism if for each $\bx \in \mathcal{X}$, a neighborhood of $\bx$ is mapped homeomorphically by $\bH$ onto a neighborhood of $\bH(\bx)$.
\end{definition}
\begin{remark}
    It is clear that if $\bH : \Real^n \rightarrow \Real^n$ is a continuously differentiable mapping with non-vanishing $\det(\Diff \bH(\bx))$, it follows from the inverse function theorem that $\bH$ is a local homeomorphism.
\end{remark}

When continuous mapping $\bH : \Real^n \rightarrow \Real^n$ satisfies 
\begin{align*}
    \|\bH(\bx)\| \longrightarrow \infty \quad\text{as}\quad \|\bx\| \longrightarrow \infty,
\end{align*}
we refer to it as a proper mapping \cite{Parthasarathy1983}.

\begin{proposition}\cite{Parthasarathy1983}\label{thm:Hadamard}
    Mapping $\bH: \mathbb{R}^{n} \rightarrow \mathbb{R}^{n}$ is a homeomorphism if and only if $\bH$ is a proper mapping and a local homeomorphism.
\end{proposition}


\begin{proposition}\cite{Parthasarathy1983}\label{lem:parth_injective}
Let $\bH: \mathbb{R}^n \rightarrow \mathbb{R}^n$ be continuously differentiable and $\Omega$ be a bounded rectangle in $\mathbb{R}^n$. Suppose that  $\det(\Diff \bH(\bx))>0$ for all $ \bx \in\mathbb{R}^n$ and further suppose $\Diff \bH(\bx)$ is a weak P-matrix for all $\bx \in \mathbb{R}^n \backslash \Omega$. Then $F$ is one-to-one (injective).
\end{proposition}

\begin{proposition}\cite{Gale1965}\label{lem:gale_R2_injective}
    Let $\Omega\subset \Real^{2}$ be an arbitrary rectangle either closed or not closed and $\bH: \Omega \rightarrow \Real^{2}$. Suppose that $\bH$ has continuous partial derivatives and none of the principal minors of $\Diff \bH(\bx)$ 
    vanishes for all $\bx\in\Omega$. Then $\bH$ is injective in $\Omega$.
\end{proposition}

\begin{proposition}\cite{guillemin2010differential}\label{lem:general_IFT}
    Let $\bH: U \rightarrow V$ be a mapping between open subsets of $\mathbb{R}^n$ and $\mathbb{R}^m$. Assume $\bH$ is continuously differentiable. If $\bH$ is injective on a closed subset $\Omega \subset U$ and if the Jacobian matrix of $\bH$ is invertible at each point of $\Omega$, then $\bH$ is injective in a neighborhood $\Omega^{\prime}$ of $\Omega$ and $\bH^{-1}: \bH\left(\Omega^{\prime}\right) \rightarrow \Omega^{\prime}$ is continuously differentiable.
    \end{proposition}
 

\begin{proposition}\label{lem:cauchy_binet}\cite[Cauchy--Binet theorem]{BRUALDI1983}
    Let $A$ and $B$ be $K\times M$ matrices, where $K\le M$, then
\begin{align}
\det (AB^{\tra}) =& \sum_{\#\beta = K} \det(A[\langle K \rangle,\beta]) \det(B[\langle K \rangle,\beta]).
\end{align}
\end{proposition}

\section{Mathematical analysis}\label{sec:math_analy}

In this section, we analyze the existence, uniqueness, and stability of the solution to nonlinear system \cref{eq:nonlinear_systems_map}. The Jacobian matrix of mapping $\bF$ can be calculated that is given by
\begin{align}\label{eq:jacobian_matrix}
\Diff \bF(\bx) = 
    -\Lambda(\bx)G(\bx), 
    \end{align}
where  
\begin{align*}
\Lambda(\bx) &= \diag\left(\frac{1}{\sum_{m=1}^{M} s_{m}^{[1]} \zeta_{m}(\bx)}, \ldots, \frac{1}{\sum_{m=1}^{M} s_{m}^{[Q]} \zeta_{m}(\bx)}\right),\\[2mm]
G(\bx) &= \begin{bmatrix}
        \sum_{m=1}^M s^{[1]}_m \zeta_{m}(\bx) b_{1m} & \cdots & \sum_{m=1}^M s^{[1]}_m \zeta_{m}(\bx) b_{Km}\\
        \vdots & \ddots &\vdots \\
        \sum_{m=1}^M s^{[Q]}_m \zeta_{m}(\bx) b_{1m} & \cdots & \sum_{m=1}^M s^{[Q]}_m \zeta_{m}(\bx) b_{Km}
    \end{bmatrix}, \\[2mm]
\zeta_{m}(\bx) &=\exp\left(-\sum_{k=1}^K b_{km} x_{k}\right).
\end{align*} 
We define
\begin{align*}
    S = [\bds^{[1]}, \ldots, \bds^{[Q]}]^{\tra} \in \Real^{Q\times M}, \quad
    B = [\bdb^{[1]}, \ldots, \bdb^{[K]}]^{\tra} \in \Real^{K\times M}, 
\end{align*}
where $\bds^{[q]} = [s^{[q]}_1, \ldots, s^{[q]}_M]^{\tra} \in \Real^M$, and $\bdb^{[k]} = [b_{k1}, \ldots, b_{kM}]^{\tra} \in \Real^M$ denote the $q$-th row of matrix $S$, and the $k$-th row of matrix $B$, respectively. 
Subsequently, we can rewrite   
\begin{align*}
G(\bx) 
& = S\diag(\bzeta(\bx))B^{\tra}, 
\end{align*} 
where $\bzeta(\bx) = [\zeta_{1}(\bx), \ldots, \zeta_{M}(\bx)]^{\tra}$. 

We also make the following assumption, which appears appropriate for practical conditions considered in the work:
\begin{assumption}\label{ass:s_b}
   $S\ge 0$ and $B>0$, namely, $S$ is a non-negative matrix, and $B$ is a positive matrix. 
    Assume further that all columns of $S$ are non-zero vectors and 
\[
\sum_{m=1}^M s_m^{[q]} = 1 \quad \text{for}\ 1\le q \le Q, 
\]
and $Q = K\le M$.
\end{assumption}

\subsection{Existence}
The existence of a solution for nonlinear system \cref{eq:nonlinear_systems_map} is equivalent to that for data $\bdg\in\Real^Q$, there is solution $\bx$ such that $\bF(\bx)=\bdg$. Namely, $\bF$ is a surjection in $\Real^Q$. The surjectivity is proven below by means of that $\bF$ is shown to be a homeomorphism. 

Firstly, we calculate the principal minor of $G(\bx)$. 
\begin{lemma}\label{lem:Jacobian_det} 
The principal minor of $G(\bx)$ in \cref{eq:jacobian_matrix} is given by
\begin{align}\label{eq:det_jacobian}
\det\bigl(G(\bx)[\alpha]\bigr) 
    = \sum_{\#\beta=\#\alpha} \left(\prod_{i\in\beta}\zeta_{i}(\bx)\right)
    \det(S[\alpha,\beta]) 
    \det(B[\alpha,\beta]). 
    \end{align}
\end{lemma}
\begin{proof}
The principal submatrix of $G(\bx)$ can be rewritten as 
\begin{align*}
    G(\bx)[\alpha] = S[\alpha,\langle M \rangle] \diag(\bzeta(\bx)) B[\alpha,\langle M \rangle]^{\tra}.
\end{align*}
By \cref{lem:cauchy_binet}, the result in \cref{eq:det_jacobian} is immediately obtained. 
\end{proof}

We then give the following theorem. 
\begin{theorem}\label{thm:local_homeo}
    Assume that $\det(SB^{\tra})\neq 0$ and for  index set $\beta \subseteq \langle M \rangle$, $\#\beta=Q$, the products 
    of the following minors of $S$ and $B$ satisfy 
    \begin{equation}\label{eq:det_same_sign}
        \det(S[\langle Q \rangle,\beta]) 
        \det(B[\langle Q \rangle,\beta]) \ge 0 \ (\text{or}\ \le0), 
    \end{equation}
    namely, all of these products are non-negative (or non-positive). Then $\bF$ defined in \cref{eq:nonlinear_systems_map} is a local homeomorphism in $\Real^{Q}$.
    \end{theorem}
    \begin{proof}
    Considering $\bF$ is continuously differentiable in $\Real^Q$,
    \cref{eq:jacobian_matrix}, and \cref{lem:Jacobian_det}, we have
\begin{multline*}
    \det\bigl(\Diff \bF(\bx)\bigr) = (-1)^Q
    \left(\prod_{q = 1}^Q \langle \bds^{[q]}, \bzeta(\bx) \rangle\right)^{-1}\\
     \sum_{\#\beta=Q}\left(\prod_{i \in \beta}\zeta_{i}(\bx)\right)
    \det(S[\langle Q \rangle,\beta]) 
    \det(B[\langle Q \rangle,\beta]). 
\end{multline*}
   In particular, when $\bx=\bzero$, we have
\begin{align*}
    \det\bigl(\Diff \bF(\bzero)\bigr) = (-1)^Q \sum_{\#\beta=Q} \det(S[\langle Q \rangle,\beta]) 
    \det(B[\langle Q \rangle,\beta]) = (-1)^Q \det(SB^{\tra}).
\end{align*}
    As $\det(SB^{\tra})\neq 0$, there is at least one index such that its corresponding product defined as \cref{eq:det_same_sign} is non-vanishing. Hence we have
    \begin{align*}
        \det\bigl(\Diff \bF(\bx)\bigr)\neq 0, \quad \text{for}\ \bx\in\Real^Q.
    \end{align*}
Consequently, using \cref{def:local_homeo}, we can show that $\bF$ is a local homeomorphism in $\Real^{Q}$.
\end{proof}

%
%


Furthermore, we define an index set 
\begin{align}\label{eq:def_M_index_set}
    \mathcal{M}(k,l) := \mathop{\arg\max}\limits_{m} \left\{\frac{b_{{k}{m}}}{b_{{l}{m}}} \right\},
\end{align} 
where $k,l\in \langle K\rangle, k\neq l$.

\begin{theorem}\label{thm:not_proper}
    If there exists indices $k_1, k_2$, and $m_0\in\mathcal{M}(k_1,k_2)$ such that $s_{m_0}^{[q]}>0$ for $1\le q\le Q$,  $\bF$ defined in \cref{eq:nonlinear_systems_map} is not a proper mapping.
\end{theorem}
\begin{proof}
    Let \begin{align*}
        \bar{b}:=\max_{m} \frac{b_{{k_1}{m}}}{b_{{k_2}{m}}} = \frac{b_{{k_1}{m_0}}}{b_{{k_2}{m_0}}}.
    \end{align*} 
    Taking
    \begin{align*}
        \bx=[0,\ldots,\underbrace{-r}_{k_{1}},0,\ldots,0,\underbrace{r\bar{b}}_{k_{2}},\ldots, 0]^{\tra},
    \end{align*}
where $r>0$, we then obtain the $m$-th component of $B^{\tra}\bx$ as
\begin{align*}
    (B^{\tra}\bx)_m = \sum_{k=1}^K b_{km} x_k &= -b_{{k_1}m}r + b_{{k_2}m}r\bar{b} \\
    & \ge r \left(-b_{{k_1}m} + b_{{k_2}m} \frac{b_{{k_1}{m}}}{b_{{k_2}{m}}}\right) \\
    & = 0,
\end{align*}
For $m_0$, we have $(B^{\tra}\bx)_{m_0} = r(-b_{{k_1}{m_0}} + b_{{k_2}{m_0}}\bar{b})=0$. Moreover, $s_{m_0}^{[q]}>0$ for all $1\le q\le Q$. Therefore, we obtain
\begin{align*}
    \lim_{\|\bx\|\rightarrow\infty} \exp(-B^{\tra}\bx) = \lim_{r\rightarrow\infty} \exp(- B^{\tra}\bx)=[0,\ldots,\underbrace{1}_{m_0},\ldots, 0]^{\tra}.
\end{align*}
There may exist $m_i\in \mathcal{M}(k_1,k_2)$ implying $(B^{\tra}\bx)_{m_i}=0$, and there may be other non-vanishing components in the above limit vector. Therefore, we have
    \begin{align*}
        \lim_{\|\bx\|\rightarrow\infty} \|\bF(\bx)\| = \lim_{r\rightarrow\infty} \Big\|\ln \begin{bmatrix}
            \langle \bds^{[1]}, \exp(-B^{\tra}\bx) \rangle \\
            \vdots \\
            \langle \bds^{[Q]}, \exp(-B^{\tra}\bx) \rangle 
        \end{bmatrix}\Big\| 
        = \Big\| \ln \begin{bmatrix}
            \sum_{i} s_{m_i}^{[1]}\\
            \vdots \\
            \sum_{i} s_{m_i}^{[Q]}
        \end{bmatrix}\Big\|.
    \end{align*}
Considering \cref{ass:s_b}, we now have
\begin{align}
    \lim_{\|\bx\|\rightarrow\infty} \|\bF(\bx)\| \le +\infty,
\end{align}
which implies that $\bF$ is not a proper mapping.
\end{proof}

For the more general case, we have the following result:
\begin{corollary}\label{thm:general_not_proper}
 For any $\bx\in\Real^Q$ and indices $m_1,m_2,\ldots,m_l$ such that 
    \begin{align*}
        B^{\tra}\bx \ge0 \quad \text{and}\quad (B^{\tra}\bx)_{m_i} = 0,
    \end{align*}
    where $1\le i \le l$. If there exists $m_i$ such that  
    $s_{m_i}^{[q]}>0$ for  $1\le q\le Q$, then $\bF$ defined in \cref{eq:nonlinear_systems_map} is not a proper mapping.
\end{corollary}
\begin{proof}
    The proof is similar to that of \cref{thm:not_proper}.
\end{proof}


The result in \cref{thm:not_proper} implies a necessary condition for that $\bF$ is a 
proper mapping. We claim that this condition is sufficient and necessary for the case of two-dimensional  DECT.

\begin{theorem}\label{thm:DECT_proper}
Let $Q=K=2$, then $\bF$ defined in \cref{eq:nonlinear_systems_map} is a proper mapping if and only if there exists index $q:=q(k,l)$ such that $s_{m}^{[q]}=0$ for any $m\in\mathcal{M}(k,l)$ (see \cref{eq:def_M_index_set} for its definition), where $k, l\in\{1,2\}$, $k\neq l$.
\end{theorem}

\begin{proof}
$``\Longrightarrow"$.
If the condition does not hold, without loss of generality, there exists $m_0\in\mathcal{M}(1,2)$, $s_{m_0}^{[q]}>0$ 
for $1\le q\le Q$. By \cref{thm:not_proper}, $\bF$ is not a proper mapping, which is a contradiction. 

$``\Longleftarrow"$.  
Using polar coordinates $\bx = r\bdn$, where $r>0$, $\bdn = [\cos\theta,\sin\theta]^{\tra}$, 
and $\theta\in[0,2\pi)$, we can obtain
\begin{align*}
    (B^{\tra}\bx)_m = \sum_{k=1}^2 b_{km}x_k &= b_{1m} r\cos\theta + b_{2m} r\sin\theta \\
    &= r\sqrt{b_{1m}^2 + b_{2m}^2} \sin(\theta+\theta_m),
\end{align*}
where $\theta_m = \arctan (b_{1m}/b_{2m})$. Assume that $\{\theta_m\}$ have the following order 
\begin{align*}
    \theta_{1}= \cdots =\theta_{l} > \theta_{l} \ge \cdots \ge \theta_{p} >\theta_{p}=\cdots= \theta_{M}. 
\end{align*}
Otherwise, they can be reordered. 
Thus the condition can be translated into that there exists $q_1$ such that $s_{m}^{[q_1]}=0$ for $m\in\mathcal{M}(1,2)=\{1,2,\ldots,l\}$, and there exists $q_2$ such that $s_{n}^{[q_2]}=0$ for $n\in\mathcal{M}(2,1)=\{p+1,p+2,\ldots,M\}$. Next, according to the value of angle $\theta$, the proof can be split into three cases.
\paragraph{Case I:} $0\le \theta <\pi-\theta_{1}+\frac{\theta_{1}-\theta_{l+1}}{2}$. For any $m\in\mathcal{M}(1,2)^c = \langle M \rangle \backslash \mathcal{M}(1,2)$, we have 
\begin{align*}
    0<\theta_M \le \theta_m \le \theta + \theta_m < \pi-\theta_{1}+\frac{\theta_{1}-\theta_{l+1}}{2} + \theta_m \le \pi-\frac{\theta_{1}-\theta_{l+1}}{2}<\pi.
\end{align*}
Since $(B^{\tra}\bdn)_m = \sqrt{b_{1m}^2 + b_{2m}^2} \sin(\theta+\theta_m)$, there exists $c_1>0$ such that 
\begin{align*}
    (B^{\tra} \bdn)_m \ge c_1,\quad \forall \ m \in \mathcal{M}(1,2)^c.
\end{align*}
 Consequently, when $r$ becomes sufficiently large,
\begin{align*}
   | F_{q_1}(\bx) | &= | \ln \langle \bds^{[q_1]}, \exp(-rB^{\tra}\bdn) \rangle |\\
   &= - \ln \sum_{m\in\mathcal{M}(1,2)^c} s_m^{[q_1]}\exp(-r(B^{\tra}\bdn)_m)  \\
   & \ge  c_1 r -\ln\sum_{m\in\mathcal{M}(1,2)^c} s_m^{[q_1]}  \\
   & = c_1 r.
\end{align*}
The last line comes from the condition that $s_{m}^{[q_1]}=0$ for any $m\in\mathcal{M}(1,2)$, and $\sum_m s_m^{[q_1]}=1$.

\paragraph*{Case II:}$\pi-\theta_{1}+\frac{\theta_{1}-\theta_{l+1}}{2}\le\theta< 2\pi-\theta_{M}-\frac{\theta_{p}-\theta_{M}}{2}$.
When $\pi-\theta_{1}+\frac{\theta_{1}-\theta_{l+1}}{2}\le\theta < \pi-\theta_M$, considering the range of $\theta+\theta_1$, we have
\begin{align*}
    \pi < \pi +\frac{\theta_{1}-\theta_{l+1}}{2} \le \theta + \theta_1< \pi-\theta_M +\theta_1 < 2\pi.
\end{align*}
When $\pi-\theta_M \le \theta < 2\pi-\theta_1$, considering the range of $\theta+\theta_p$, we have
\begin{align*}
    \pi < \pi-\theta_M +\theta_p \le \theta + \theta_p < 2\pi-\theta_1+\theta_p < 2\pi.
\end{align*}
When $2\pi-\theta_1 \le \theta < 2\pi-\theta_{M}-\frac{\theta_{p}-\theta_{M}}{2}$, considering the range of $\theta+\theta_M$, we have
\begin{align*}
    \pi< 2\pi-\theta_1 +\theta_{M} \le  \theta + \theta_{M} < 2\pi-\frac{\theta_{p}-\theta_{M}}{2} < 2\pi.
\end{align*}
Thus  there exists $c_2>0$ and index $m({\theta})$ such that 
\begin{align*}
    (B^{\tra} \bdn)_{m(\theta)} \le -c_2.
\end{align*}
Under \cref{ass:s_b}, all columns of $S$ are nonzero vectors, suggesting that there always exists $q$ such that $s_{m({\theta})}^{[q]}>0$. When $r$ becomes sufficiently large,
\begin{align*}
    | F_{q}(\bx) | &= | \ln \langle \bds^{[q]}, \exp(-rB^{\tra}\bdn) \rangle |\\
   &\ge  \ln s_{m({\theta})}^{[q]}\exp(-r(B^{\tra}\bdn)_{m({\theta})})  \\
   & \ge \tilde{c}_2 r - |\ln s_{m({\theta})}^{[q]}| \\
   & \ge c_2 r.
\end{align*}
Note that for given $S$, $|\ln s_{m({\theta})}^{[q]}|$ is bounded under \cref{ass:s_b}.

\paragraph*{Case III:}$2\pi-\theta_{M}-\frac{\theta_{p}-\theta_{M}}{2}\le \theta<2\pi$. In this case, for any $m\in\mathcal{M}(2,1)^c$, we have
\begin{align*}
    0<\frac{\theta_{p}-\theta_{M}}{2} \le -\theta_{M}-\frac{\theta_{p}-\theta_{M}}{2} + \theta_m \le (\theta + \theta_m) -2\pi < \theta_m <\theta_1 <\frac{\pi}{2}. 
\end{align*} 
There exists  then $c_3>0$ such that 
\begin{align*}
    (B^{\tra} \bdn)_m \ge c_3,\quad \forall  m \in \mathcal{M}(2,1)^c.
\end{align*}
Similarly, when $r$ becomes sufficiently large,
\begin{align*}
    | F_{q_2}(\bx) | &= | \ln \langle \bds^{[q_2]}, \exp(-rB^{\tra}\bdn) \rangle |\\
    &= - \ln \sum_{m\in\mathcal{M}(2,1)^c} s_m^{[q_2]}\exp(-r(B^{\tra}\bdn)_m)  \\
    & \ge c_3 r - \ln\sum_{m\in\mathcal{M}(2,1)^c} s_m^{[q_2]}  \\
    & = c_3 r.
 \end{align*} 

We have shown that for any $\bx\in\Real^Q$,
\begin{align*}
    \|\bF(\bx)\| \ge \frac{1}{\sqrt{2}} \sum_{q} |F_{q}(\bx)| \ge \frac{1}{\sqrt{2}}  \min \{c_1, c_2, c_3\} r = c\|\bx\|,
\end{align*}
where $c := \min \{c_1, c_2, c_3\}/\sqrt{2}$.
Finally, we obtain
\begin{align*}
    \lim_{\|\bx\|\rightarrow +\infty} \|\bF(\bx)\| = +\infty,
\end{align*}
which completes the proof.
\end{proof}

\begin{remark}
If $s_{m_0}^{[q]}=0$ for all $q$, then taking $\bx=[0,\ldots,r,\ldots,0]^{\tra}$, 
we have $\lim\limits_{\|\bx\|\rightarrow\infty} \|\bF(\bx)\|=\sqrt{2}\nrightarrow \infty$, 
so $\bF(\bx)$ is not a proper mapping. Hence, it is necessary that all columns of $S$ are non-vanishing 
as indicated in \cref{ass:s_b}. 
\end{remark}

Combining with \cref{thm:Hadamard}, we can further claim that $\bF$ is a homeomorphism on $\Real^2$. Consequently, the existence and uniqueness condition for the solution of nonlinear system $\bF(\bx) = \bdg$ is given by the following theorem.


\begin{theorem}\label{thm:DECT_homeo}
    Let $Q=K=2$. Suppose that the conditions in \cref{thm:local_homeo} hold. Furthermore, there exists index $q:=q(k,l)$ such that $s_{m}^{[q]}=0$ for any $m\in\mathcal{M}(k,l)$, where $k, l\in\{1,2\}$, $k\neq l$. Then, $\bF$ defined in \cref{eq:nonlinear_systems_map} is a homeomorphism. 
\end{theorem}
\begin{proof}
    By \cref{thm:local_homeo}, $\bF$ is a local 
    homeomorphism, and by \cref{thm:DECT_proper}, $\bF$ is a proper mapping. 
    Using \cref{thm:Hadamard}, we obtain the desired result.
\end{proof}

\begin{remark}\label{remark:2D_diffeo}
    The author in \cite{Gordon1972OnTD} proved a result that continuously differentiable mapping $\bH$ from $\Real^n$ to $\Real^n$ is a diffeomorphism if and only if $\bH$ is a proper mapping and its Jacobian determinant $\det(\Diff \bH)$ never vanishes. Based upon this theorem, we conclude that the mapping defined in \cref{eq:nonlinear_systems_map} is a diffeomorphism under the conditions in \cref{thm:DECT_homeo}.
\end{remark}

\subsection{Uniqueness}
We assume the nonlinear system in \cref{eq:nonlinear_systems_map} has at least a solution. We will analyze the uniqueness of the solution. First of all, we have the following lemma. 
\begin{lemma}\label{lem:linear_invertible}
Let $\bH, \tilde{\bH}: \Real^{n} \rightarrow \Real^{n}$ be two mappings such that $\tilde{\bH}=A_1 \circ \bH \circ A_2$ where $A_1,A_2: \mathbb{R}^{n} \rightarrow$ $\Real^{n}$ are invertible linear transformations. Then, $\tilde{\bH}$ is injective if and only if $\bH$ is injective. 
\end{lemma}
\begin{proof}
    The proof is similar to that of proposition 5 in \cite{Bal_2020}.
\end{proof}

Next, we give a sufficient condition for the globally injective property of mapping $\bF$.

\begin{theorem}\label{thm:univalence_SPCT}
Assume that $\det(SB^{\tra})\neq 0$ and for all indexes $\alpha\subseteq \langle Q \rangle$, 
$\beta \subseteq \langle M \rangle$, $\#\beta=\#\alpha$, the products of the following minors of $S$ and $B$ satisfy  
\begin{equation}\label{eq:same_sign}
    \det(S[\alpha,\beta]) 
    \det(B[\alpha,\beta]) \ge 0, 
\end{equation}
\ie, all of these products are non-negative. Then $\bF$ defined in \cref{eq:nonlinear_systems_map} is globally injective in $\Real^{Q}$.
\end{theorem}
\begin{proof}
Firstly, we consider mapping $-\bF$. By \cref{eq:jacobian_matrix} and \cref{lem:Jacobian_det}, for any principal minor of  $\Diff(-\bF)(\bx)$
\begin{multline*}
\det\bigl(\Diff (-\bF)(\bx)[\alpha]\bigr) =
\det\bigl(\Lambda(\bx)[\alpha]\bigr) \det\bigl(G(\bx)[\alpha]\bigr) = 
\left(\prod_{q \in \alpha} \langle \bds^{[q]}, \bzeta(\bx) \rangle\right)^{-1} \\
  \sum_{\#\beta=\#\alpha}\left(\prod_{i \in \beta}\zeta_{i}(\bx)\right)
  \det(S[\alpha,\beta]) 
  \det(B[\alpha,\beta]). 
\end{multline*}
For $\alpha = \langle Q \rangle$, recalling the proof of \cref{thm:local_homeo}, 
we have $\det\bigl(\Diff (-\bF)(\bx)\bigr)>0$. For $\alpha\subset \langle Q \rangle$, by \cref{eq:same_sign}, $\det\bigl(\Diff (-\bF)(\bx)[\alpha]\bigr)$ is non-negative for $\bx\in\Real^Q$.
By \cref{lem:parth_injective}, we know that $-\bF$ is an injective mapping, 
and then using \cref{lem:linear_invertible}, $\bF$ is also injective.
\end{proof}
\begin{remark}
    Note that the sign condition in \cref{eq:same_sign} is well-defined since it is automatically true when $B=S$.
\end{remark}

In particular, when $Q=K=2$, by the conditions in \cref{thm:local_homeo}, we can also prove that $\bF$ is global injective in $\Real^2$.  This corresponds to the case of DECT.

\begin{theorem}\label{thm:univalence_DECT}
Let $Q=K=2$. Assume that the conditions in \cref{thm:local_homeo} hold.
Then $\bF$ defined in \cref{eq:nonlinear_systems_map} is globally injective in $\Real^{2}$.
    \end{theorem}
\begin{proof}
    Similar to the proof of \cref{thm:univalence_SPCT}, we have 
    \begin{align*}
        \det (\Diff \bF(\bx)) \neq 0 \quad\text{for}\ \bx\in\Real^2.
    \end{align*}
By \cref{ass:s_b}, the first-order principal minors also never vanish for any $\bx\in\Real^2$. 
By \cref{lem:gale_R2_injective}, we obtain that $\bF$ is injective in any rectangle, which implies the desired result. 
\end{proof}

Using \cref{lem:linear_invertible}, we study whether there is invertible linear transformation $A$ such that $\Diff (-\bF \circ A)$ is a (weak) P-matrix for any $\bx\in\Real^Q$. 
In fact, we have
\begin{align*}
    \mathcal{D}(-\bF\circ A)(\bx) = \Lambda(\bx) S \diag(\bzeta(\bx))
    B^{\tra} A.
\end{align*}
By \cref{thm:univalence_SPCT}, the question is equivalent to asking whether there is  invertible linear transformation $A$ such that $S$ and $A^{\tra} B$ satisfy the sign conditions as in \cref{eq:same_sign} or \cref{eq:det_same_sign}. If it holds, $\mathcal{D}(-\bF\circ A)(\bx)$ is (weak) P-matrix for all $\bx\in\Real^Q$. 
As an example, for \cref{eq:det_same_sign}, we have the following theorem.
\begin{theorem}
    Assume that condition \cref{eq:det_same_sign} does not hold. None of such invertible linear transformation $A$ such that $S$ and $A^{\tra}B$ satisfy condition \cref{eq:det_same_sign}.
\end{theorem}

\begin{proof}
Without loss of generality, assume that there exist index sets $\beta_1, \beta_2 \subseteq \langle M \rangle$, $\#\beta_1 = \#\beta_2 =Q$ such that 
\begin{equation*}
    \det(S[\langle Q \rangle,\beta_1]) \det(B[\langle Q \rangle,\beta_1]) < 0, \quad \det(S[\langle Q \rangle,\beta_2]) \det(B[\langle Q \rangle,\beta_2]) > 0.
\end{equation*}
Notice that we have 
\begin{align*}
    \det(A^{\tra}B[\langle Q \rangle,\beta]) = \det(A) \det(B[\langle Q \rangle,\beta]),
\end{align*}
whereas $\det(A)$ is either positive or negative, which implies that the sign condition still does not hold. 
\end{proof}

\subsection{Stability}
With the injectivity of mapping $\bF$ in \cref{eq:nonlinear_systems_map} established, 
we discuss next the specific stability results of model \cref{eq:nonlinear_systems_map} in some bounded region.

\begin{theorem}\label{thm:stability_general}
    Suppose that the conditions of \cref{thm:univalence_SPCT} hold. 
    Let $\Omega$ be a closed bounded region, and assume that $\bF(\Omega)$ is convex. Then for every $\bx_1,\bx_2\in\Omega$, we have
    \begin{align*}
        \|\bx_1 - \bx_2\| \le \gamma  \|\bF(\bx_1) - \bF(\bx_2) \|.
    \end{align*}
    Here the constant $\gamma$ is given by 
    \begin{align*}
        \gamma:= Q^{\frac{Q}{2}}  \bigl(\max_{k,m} |b_{km}|\bigr) 
        \left(\min\limits_{\beta\subset \langle n\rangle,\bx\in\Omega} \frac{\prod_{i \in \beta}\zeta_{i}(\bx)}{\prod_{q = 1}^Q \langle \bds^{[q]}, \bzeta(\bx) \rangle} \right)^{-1}
        |\det\bigl( SB^{\tra} \bigr) |^{-1}.
    \end{align*}
\end{theorem}
\begin{proof}
    By \cref{thm:univalence_SPCT}, we obtain that $\bF$ is globally injective in $\Omega$. Hence, \cref{lem:general_IFT} yields that
    inverse mapping $\bF^{-1}$ is continuously differentiable. For any $\bx_1,\bx_2\in\Omega$, put $\bF(\bx_1) = \by_1$, $\bF(\bx_2) = \by_2$. Since $\bF(\Omega)$ is convex, then using mean value theorem to $\bF^{-1}$, we obtain
    \begin{align*}
        \|\bx_1 - \bx_2\| &= \|\bF^{-1}(\by_1) - \bF^{-1}(\by_2) \| \le \| \Diff\bF^{-1}(\by_t) (\by_1 - \by_2)\|\\
        &= \sqrt{\sum_{q=1}^Q \left( \Diff\bF^{-1}(\by_t)[q,\langle Q \rangle] (\by_1 - \by_2) \right)^2} \\
        &\le \sqrt{\sum_{q=1}^Q  \|\Diff\bF^{-1}(\by_t)[q,\langle Q \rangle]\|^2  \|\by_1 - \by_2\|^2}\\
        &= \| \Diff\bF^{-1}(\by_t) \|_{\text{F}} \|\by_1 - \by_2\| \\
        &= \| \left(\Diff\bF(\bx_t)\right)^{-1} \|_{\text{F}} \|\by_1 - \by_2\| \\    
        &= |\det(\Diff\bF(\bx_t))|^{-1} \| \text{adj}(\Diff\bF(\bx_t))\|_{F} \|\by_1 - \by_2\|,
    \end{align*}
    where $\by_t = t\by_1 +(1-t)\by_2$ with $0\le t\le 1$, $\bx_t = \bF^{-1}(\by_t)$, and $\text{adj}(\Diff\bF(\bx_t))$ denotes the adjugate matrix of $\Diff\bF(\bx_t)$. Using an estimate for the Frobenius norm of adjugate matrix \cite{Mirsky1956NormAdj}, we obtain
    \begin{align*}
        \| \text{adj}(\Diff\bF(\bx_t))\|_{F} \le Q^{\frac{Q-2}{2}} \|\Diff\bF(\bx_t) \|_{F} \le Q^{\frac{Q-2}{2}} Q \max_{k,m} |b_{km}| =  Q^{\frac{Q}{2}} \max_{k,m} |b_{km}|.
    \end{align*}
    The last inequality comes from the observation that the Jacobian of $\bF$ can be rewritten as 
    \begin{align*}
        \Diff \bF(\bx) = -\tilde{S}(\bx) B^{\tra},
    \end{align*}
    where 
    \begin{equation*}
        \tilde{S}(\bx) := \begin{bmatrix}
            \tilde{s}^{[1]}_1(\bx) & \ldots & \tilde{s}^{[1]}_M(\bx)\\
            \vdots &  & \vdots \\
            \tilde{s}^{[Q]}_1(\bx) & \ldots & \tilde{s}^{[Q]}_M(\bx)
        \end{bmatrix}
          \in\Real^{Q\times M} 
    \quad\text{with}\quad 
    \tilde{s}^{[q]}_m(\bx) := \frac{s^{[q]}_m \zeta_m(\bx) }{\sum_{m=1}^M 
    s^{[q]}_m \zeta_m(\bx)}. 
    \end{equation*}
    For the determinant of $\Diff \bF(\bx)$, by the derivation in \cref{thm:local_homeo}, we have
    \begin{multline*}
       | \det\bigl(\Diff \bF(\bx)\bigr) |= |\sum_{\#\beta=Q} \frac{\prod_{i \in \beta}\zeta_{i}(\bx)}{\prod_{q = 1}^Q \langle \bds^{[q]}, \bzeta(\bx) \rangle}
    \det(S[\langle Q \rangle,\beta]) 
    \det(B[\langle Q \rangle,\beta])| \\
    \ge \min\limits_{\beta\subset \langle n\rangle,\bx\in\Omega} \frac{\prod_{i \in \beta}\zeta_{i}(\bx)}{\prod_{q = 1}^Q \langle \bds^{[q]}, \bzeta(\bx) \rangle} |\sum_{\#\beta=Q}
    \det(S[\langle Q \rangle,\beta]) 
    \det(B[\langle Q \rangle,\beta])| \\
    = \min\limits_{\beta\subset \langle n\rangle,\bx\in\Omega} \frac{\prod_{i \in \beta}\zeta_{i}(\bx)}{\prod_{q = 1}^Q \langle \bds^{[q]}, \bzeta(\bx) \rangle} | \det\bigl( SB^{\tra} \bigr) |.
    \end{multline*}
    Hence we obtain the desired estimate.
\end{proof}

Note that, for DECT, the assumption that $F(\Omega)$ is convex is no longer required in \cref{thm:stability_general}. According to \cref{remark:2D_diffeo}, mapping $\bF$ is a diffeomorphism on $\Real^2$. Hence, for any point $\bF(\bx_1)$ and $\bF(\bx_2)$, the line segment connecting the two points lies in the range of $\bF$. Thus the mean value theorem can be used.

\section{Numerical studies}\label{sec:Numerical}

We perform quantitative studies in DECT with a parallel-beam geometry below to demonstrate numerically the conditions discussed, in which the ordinary Newton method is used for solving the discrete nonlinear system in  \cref{eq:nonlinear_systems_original}.  As shown in \cref{fig:truth_phantoms}, digital Forbild head phantom and patient-torso phantom are used in the studies, each of which consists of two basis materials $\bdf_k$ of water and bone.  MACs $b_{km}$ at energy $m$ for materials water and bone obtained from the National Institute of Standard Technology (NIST) database \cite{NIST_data}. Using \cref{eq:linear_part}, we can readily obtain truth basis sinograms $\bx_j^*$ from the truth basis images. Using software SpectrumGUI, an open-source X-ray spectrum simulator \cite{spectrum_GUI}, we first generate a pair of spectra to mimic the low-kV (\eg, 80-kV) and high-kV (\eg, 140-kV) spectra of a typical clinical CT scanner, as shown in \cref{fig:normalized_spectra}{\color{red}a}. We call it Spectra I. Additionally, we create another pair of spectra named Spectra II, consisting of the 80-kV spectrum and a filtered 140-kV spectrum, as shown in \cref{fig:normalized_spectra}{\color{red}b}. Specifically, the 80-kV spectrum is identical to the 80-kV spectrum in \cref{fig:normalized_spectra}{\color{red}a}; and the filtered 140-kV spectrum is obtained applying a copper filter of 1-mm width to the 140-kV spectrum in \cref{fig:normalized_spectra}{\color{red}a}.

We use $S_1$ and $S_2$ to denote the values of the two pairs of spectra and $B$ to denote the values of MACs obtained, as shown in  Appendix \ref{appendix:data}. 
By direct validation, $\{S_1, B\}$ and $\{S_2, B\}$ satisfy the conditions in \cref{thm:local_homeo} and \cref{thm:DECT_proper}. Using \cref{thm:DECT_homeo} and \cref{thm:univalence_DECT}, it follows that for any measured data, nonlinear system \cref{eq:nonlinear_systems_map} always has a unique solution and the $\Diff \bF(\bx)$ is always invertible. 
The two sets of low- and high-kV spectra shown in  \cref{fig:normalized_spectra}{\color{red}a} and \cref{fig:normalized_spectra}{\color{red}b} are used in the studies, respectively, involving the Forbild and torso phantoms, respectively.

\begin{figure}[htbp]
    \centering
    \includegraphics[width=1.\textwidth]{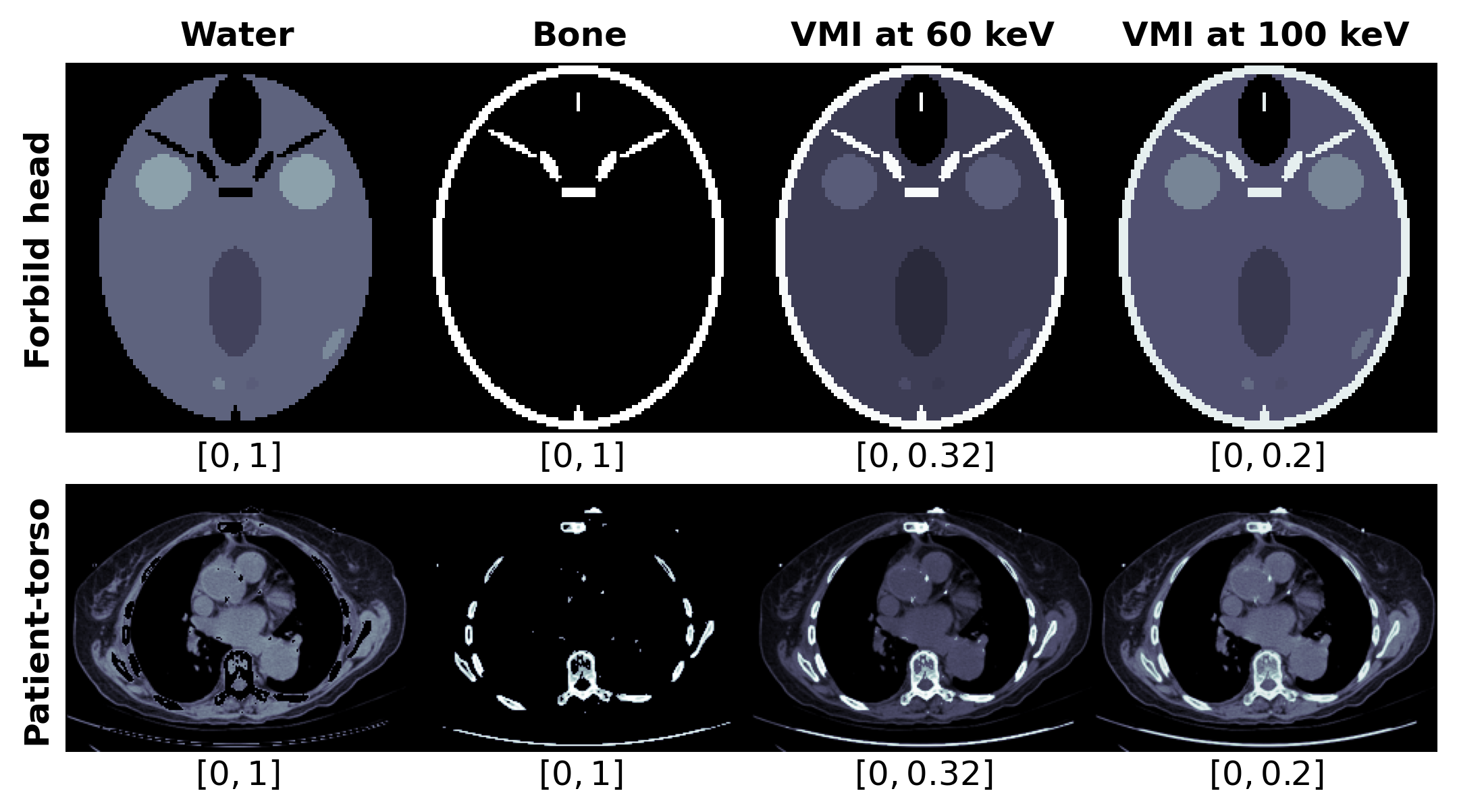}
    \caption{Truth basis images of water (column 1) and bone (column 2) and truth VMIs at energies of 60 keV (column 3) and 100 keV (column 4), respectively, of the Forbild head phantom (row 1) and the patient-torso phantom (row 2).}
    \label{fig:truth_phantoms}
\end{figure}

\begin{figure}[htbp]
    \centering
    {\includegraphics[width=0.49\textwidth]{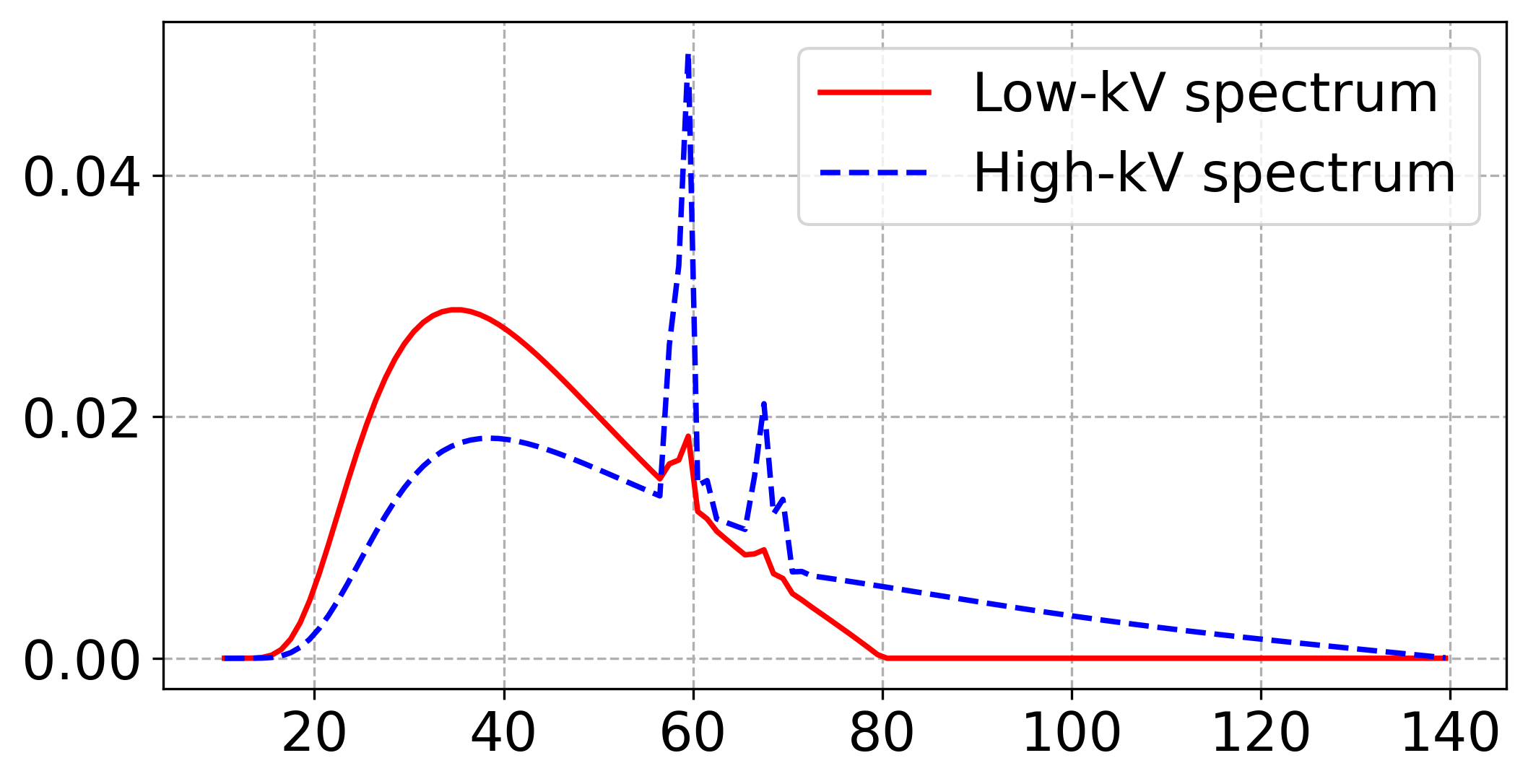}}
    {\includegraphics[width=0.49\textwidth]{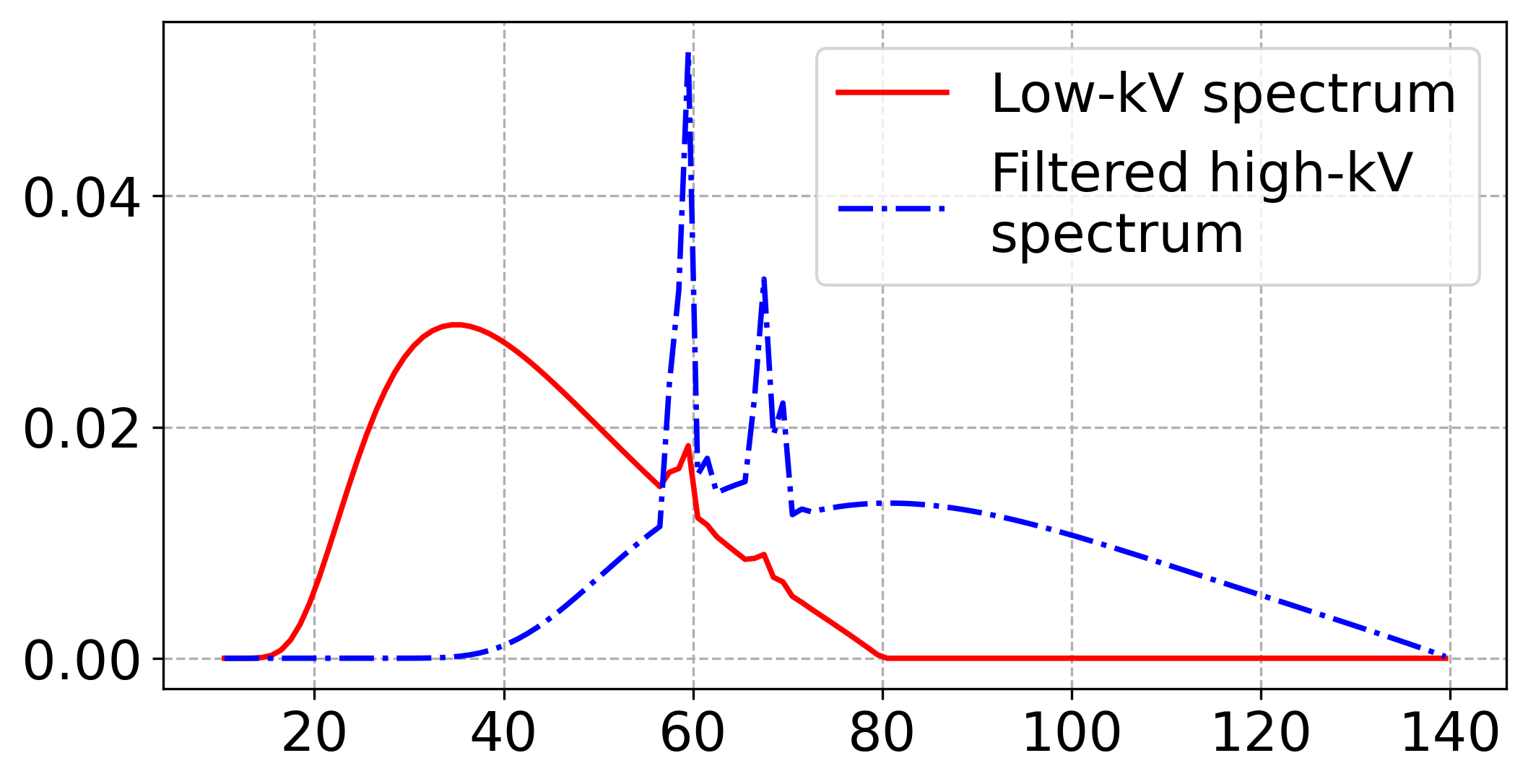}}
    \hspace{3cm} (a) \hspace{5.5cm} (b)
    \caption{(a) Spectra I: low-kV ({\color{red} red}, solid) and high-kV ({\color{blue} blue}, dashed) spectra; and (b) Spectra II: low-kV spectrum ({\color{red} red}, solid) identical to the low-kV spectrum in (a) and filtered high-kV spectrum ({\color{blue} blue}, dashed dot).}
    \label{fig:normalized_spectra}
\end{figure}

For a digital phantom, we generate noiseless data $\bdg$ by using its water and bone basis images $\bdf_k$, MACs of water and bone, and a pair of spectra in \cref{fig:normalized_spectra} in \cref{eq:nonlinear_systems_original}. From the data generated, without loss of generality, we solve iteratively the nonlinear system in \cref{eq:nonlinear_systems_map} by using the ordinary Newton method, which includes the iterative procedure \cite{Deuflhard_11} 
\begin{align}\label{eq:ordinary_newton}
    - \Diff \bF(\bx^n)\Delta \bx^n = \bF(\bx^n) - \bdg, \quad \bx^{n+1} = \bx^n + \Delta \bx^n,
\end{align}
where $n$ indicates the iteration number. We note that \cref{eq:nonlinear_systems_map} (or, equivalently, \cref{eq:nonlinear_systems_original}) can be solved independently for each ray and thus can simultaneously be solved in parallel for multiple rays.

We use the metric below to assess the recovery accuracy in the numerical studies:
\begin{align*}
    \text{RE}_{\bx}^{n} := \frac{\sum_{j=1}^J \|\bx_j^n - \bx_j^*\|^2}{\sum_{j=1}^J \|\bx_j^*\|^2},
\end{align*}
where $\bx_j^n$ denotes the basis sinogram estimated for the $j$-th ray at the $n$-th iteration, and $\bx_j^*$ the corresponding truth basis sinogram. The metric is used to 
reveal the numerical convergence of the ordinary Newton method in terms of estimation of the basis sinograms through solving nonlinear system \cref{eq:nonlinear_systems_map}.
Without loss of generality, we choose zero vector as the initial point.

Furthermore, to demonstrate numerically the stability of the solution to \cref{eq:nonlinear_systems_map}, we generate noisy data by adding white Gaussian noise to  noiseless data $\bdg$. Observing the property of a solution to \cref{eq:nonlinear_systems_map}, we can conclude that for any ray $j$,
\begin{align}\label{eq:estimate_noisy_data}
    \|\widetilde{\bx}_j^n - \bx_j^*\|&\le \|\widetilde{\bx}_j^n - \bF^{-1}(\widetilde{\bdg}_j)\| + 
    \|\bF^{-1}(\widetilde{\bdg}_j) - \bx_j^*\|  \nonumber \\ 
    &\le \underbrace{\|\widetilde{\bx}_j^n - \bF^{-1}(\widetilde{\bdg}_j)\| }_{\textit{Error 1}} + \gamma \underbrace{\| \widetilde{\bdg}_j - \bdg_j\|}_{\textit{Error 2}}, 
\end{align}
where $\widetilde{\bx}_j^n$ denotes the basis sinogram estimated for this ray at the $n$-th iteration from noisy data $\widetilde{\bdg}_j$, and $\bF^{-1}(\widetilde{\bdg}_j)$ the sinogram corresponding to $\widetilde{\bdg}_j$. 
Constant $\gamma$ is finite based upon the stability result in \cref{thm:stability_general}.  Term \textit{Error 1} is determined by the algorithm used for solving the nonlinear system, whereas \textit{Error 2} is the noise level of measured data. 

\subsection{Numerical study with the Forbild head phantom}\label{subsec:test1}

We first perform a numerical study with the 2D Forbild head phantom consisting of truth basis images of water and bone presented on an array of $128\!\times\!128$ identical pixels of square shapes covering an area $[-5, 5]\!\times\![-5, 5]$ cm$^2$,  as shown in row 1 of \cref{fig:truth_phantoms}. Using the truth basis images in \cref{eq:attenuationco-DD}, we create the truth VMIs at energies 60 keV and 100 keV shown also in row 1 of \cref{fig:truth_phantoms}. For a geometrically-consistent scan configuration considered with 181 parallel-ray projections uniformly sampled on $[-7.05, 7.05]$ cm at each of the 180 views uniformly distributed over 180$^\circ$, we can readily obtain the truth basis sinograms in column 1 of \cref{fig:sinogram_forbild} by using \cref{eq:sinoCC_formula}.  Furthermore, 
with each pair of the spectra shown in \cref{fig:normalized_spectra}, we generate a set of the noiseless data by plugging the scan geometric parameters, the low- and high-kV spectra, MACs obtained from the NIST database, and truth basis images in row 1 of \cref{fig:truth_phantoms} into  \cref{eq:dd_formula} and \cref{eq:linear_part}.
  
Using the Newton method in \cref{eq:ordinary_newton}, we solve the DD-data model in \cref{eq:nonlinear_systems_map} with each of the two sets of noiseless data generated for obtaining basis sinograms.  Metrics $\text{RE}_{\bx}^{n}$ are first displayed as functions of iteration number in \cref{fig:metric_error_forbild}{\color{red}a} for the two noiseless data sets. It can be observed that metrics numerically converge to the level of double-floating precision of the computer used. The basis sinograms obtained at iteration 10$^2$ are shown in columns 2 and 3 of \cref{fig:sinogram_forbild}.

\begin{figure}[htbp]
    \centering
    \includegraphics[width=1.\textwidth]{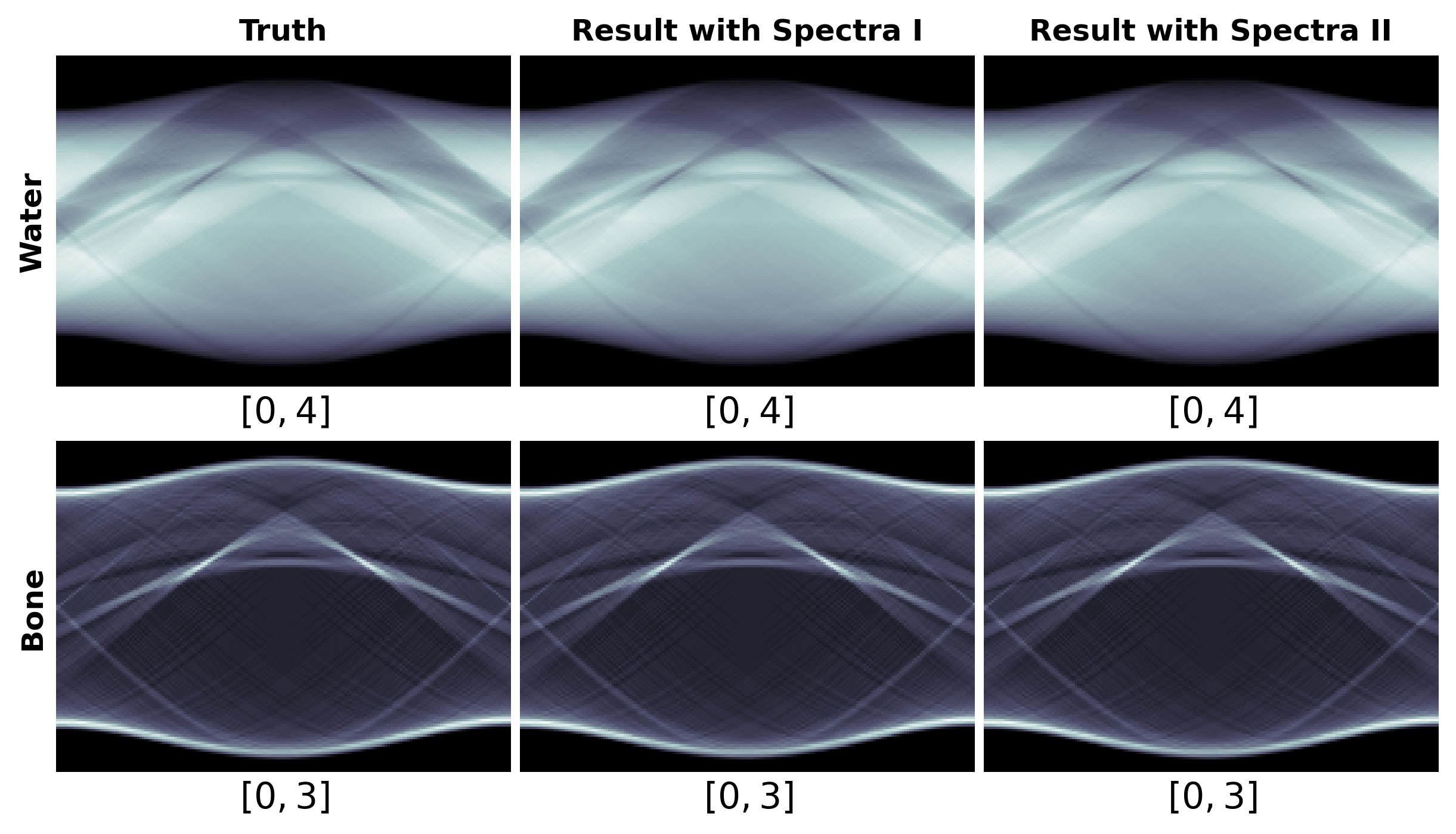}
    \caption{Truth (column 1) and estimated (columns 2 and 3) basis sinograms of water (row 1) and bone (row 2) of the Forbild head phantom from noiseless data. Columns 2 and 3 are obtained with spectral pairs in \cref{fig:normalized_spectra}{\color{red}a} and \cref{fig:normalized_spectra}{\color{red}b}, respectively.} 
    \label{fig:sinogram_forbild}
\end{figure}

\begin{figure}[htbp]
   \centering
   {\includegraphics[width=.49\textwidth]{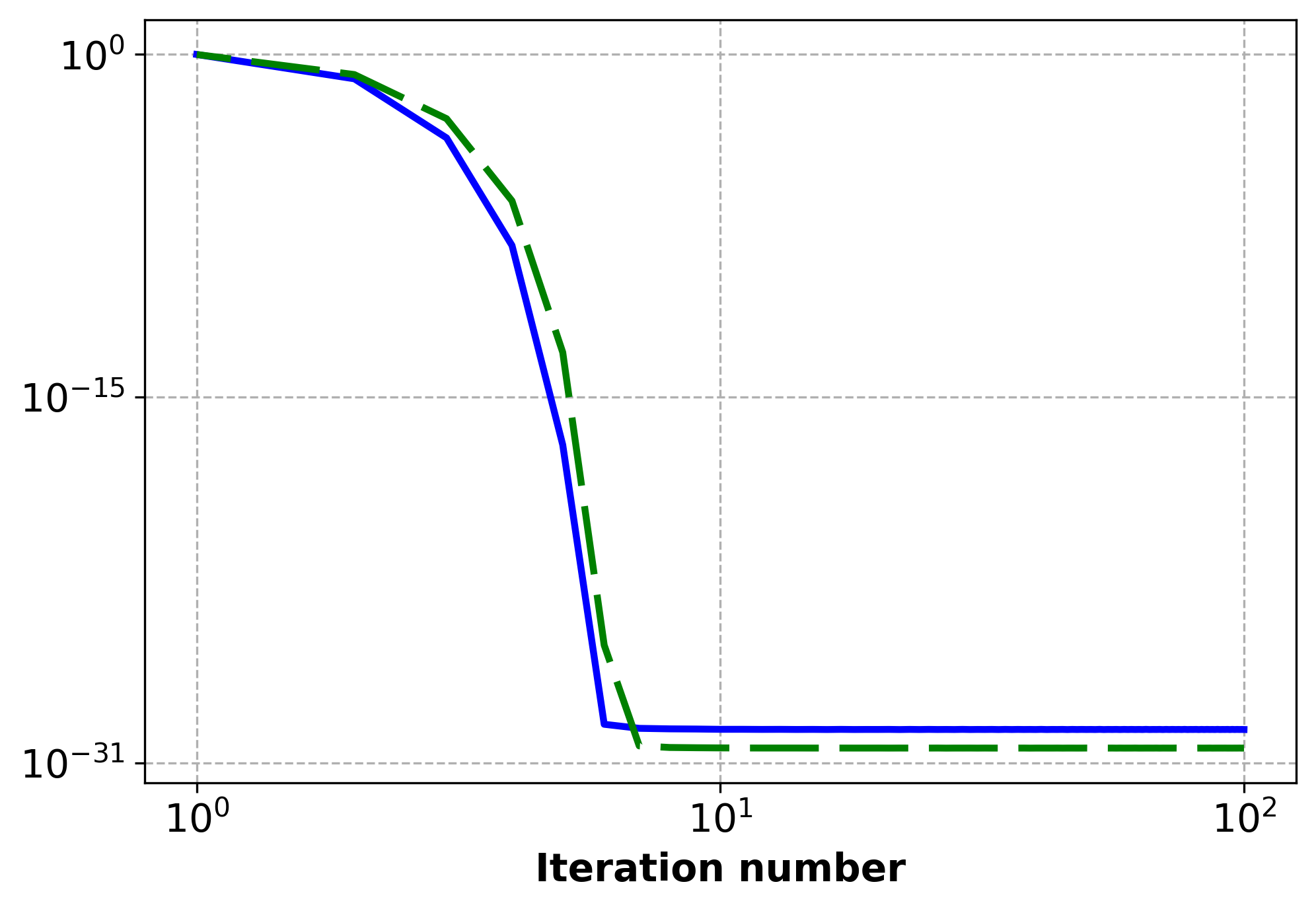}}
    {\includegraphics[width=.49\textwidth]{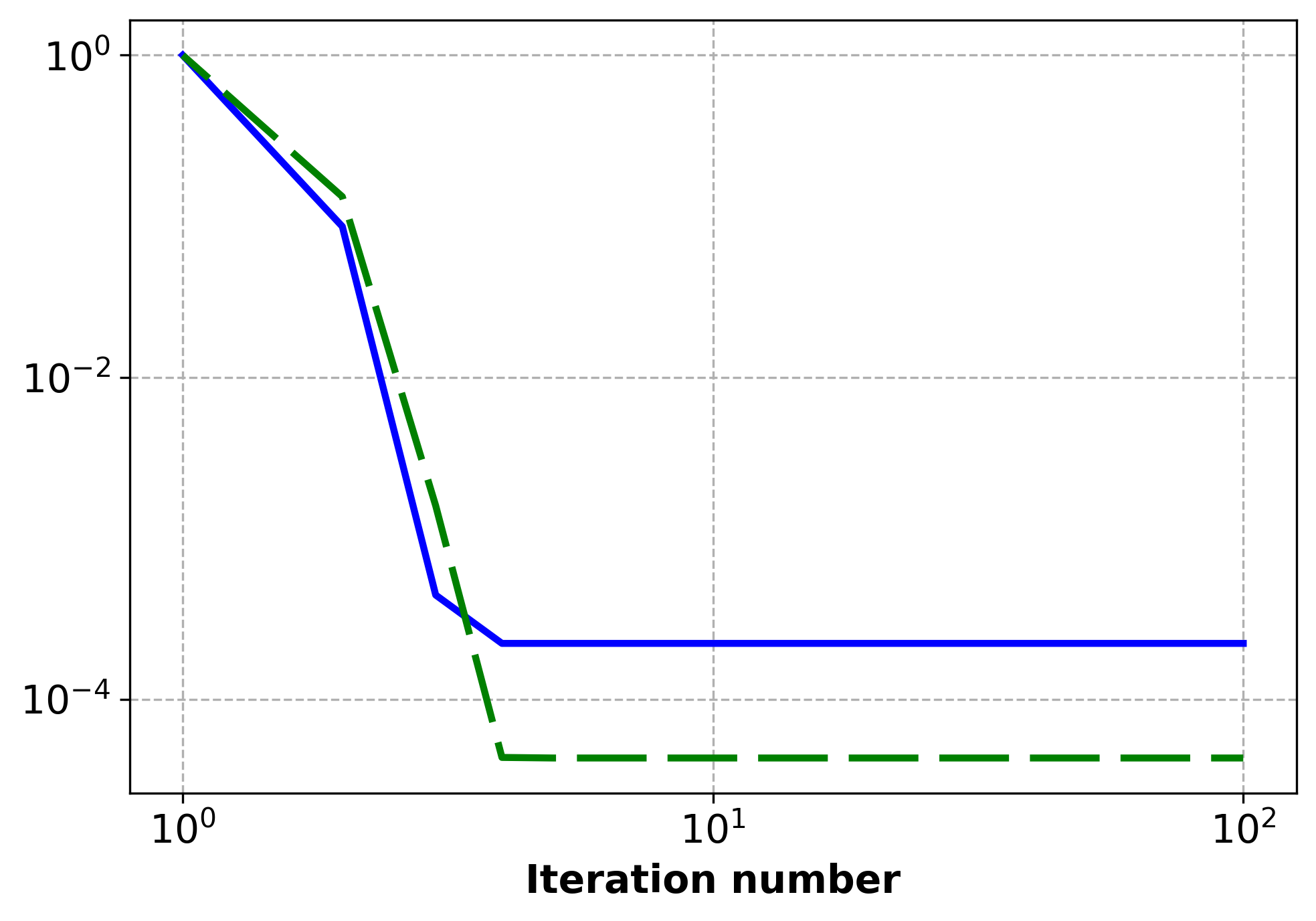}}\\
    \hspace{0.3cm} (a) \hspace{5.3cm} (b)
   \caption{Metrics $\text{RE}_{\bx}^{n}$ plotted in a log-log scale as functions of the iteration number of the ordinary Newton method obtained with the spectral pairs in \cref{fig:normalized_spectra}{\color{red}a} (solid) and \cref{fig:normalized_spectra}{\color{red}b} (dashed), respectively, from noiseless data (a) and noisy data (b) of the Forbild head phantom. 
   }
    \label{fig:metric_error_forbild}
\end{figure}

The numerical results of convergence metric $\text{RE}_{\bx}^{n}$ and basis sinograms estimated reveal that under the conditions discussed, the DD-data model in \cref{eq:nonlinear_systems_map}, can accurately be solved by use of the ordinary Newton method.  Furthermore, it can be observed in \cref{fig:metric_error_forbild} that the pair of spectra in \cref{fig:normalized_spectra}{\color{red}b} with a degree of overlapping lower than that of the other pair of spectra in \cref{fig:normalized_spectra}{\color{red}a} may be beneficial to solving the nonlinear system \cref{eq:nonlinear_systems_map}. 

We also apply the filtered backprojection (FBP) algorithm to reconstructing the basis images from the corresponding basis sinograms estimated. In \cref{fig:recon_forbild}, we display the basis images reconstructed for both pairs of spectra, and also the VMIs obtained at energies 60 keV and 100 keV with \cref{eq:attenuationco}.

\begin{figure}[htbp]
    \centering
    \includegraphics[width=0.85\textwidth]{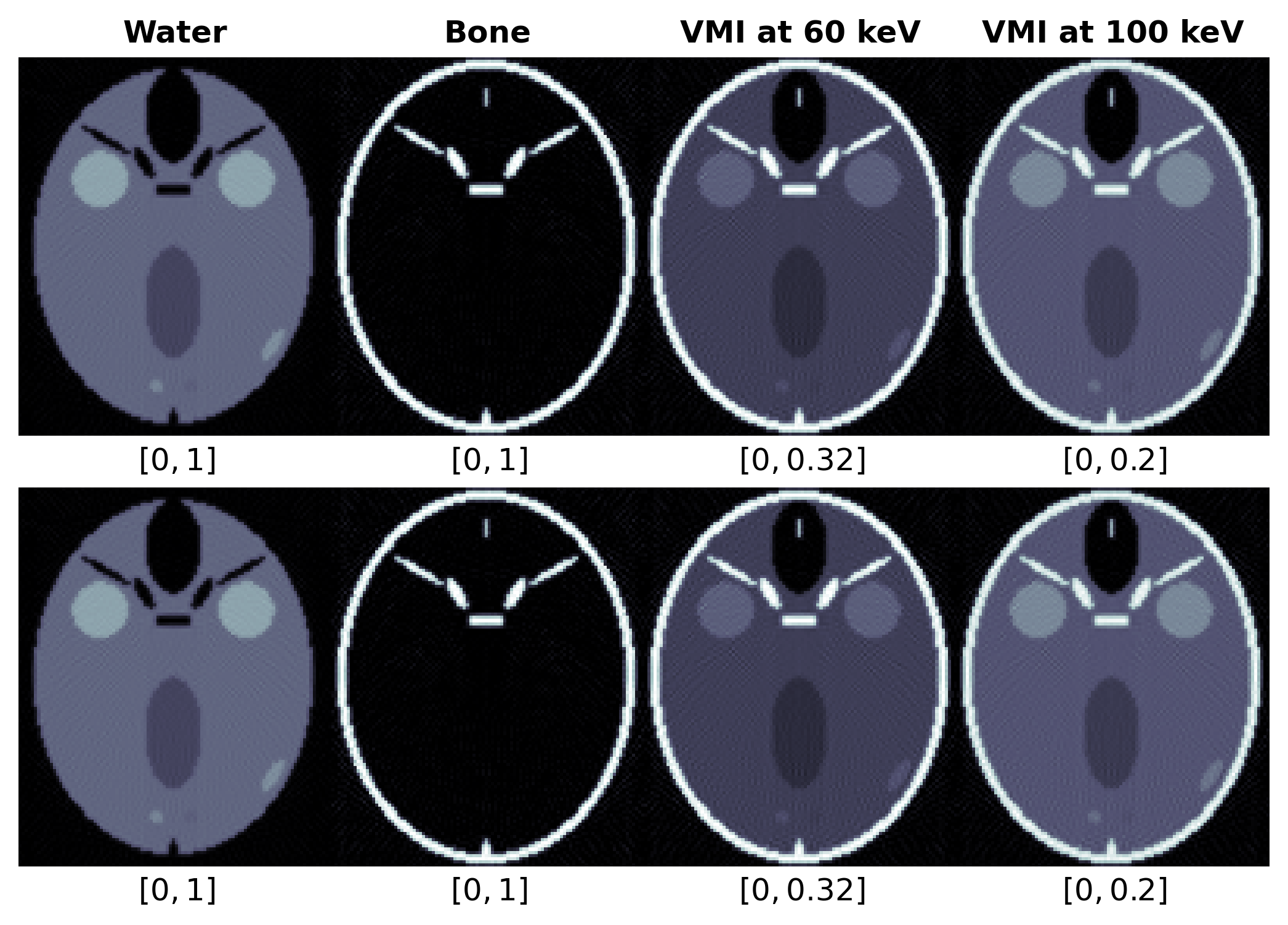}
    \caption{Basis images of water (column 1) and bone (column 2) and VMIs at energies 60 keV (column 3) and 100 keV (column 4) reconstructed from noiseless data, respectively, with the spectral pairs in \cref{fig:normalized_spectra}{\color{red}a}   (row 1) and \cref{fig:normalized_spectra}{\color{red}b} (row 2) of the Forbild head phantom, respectively.}
    \label{fig:recon_forbild}
\end{figure}

To study the numerical stability, two data sets containing Gaussian noise with signal-to-noise ratio (SNR) levels of 27.4 dB and 27.1 dB were generated for the two pairs of spectra in \cref{fig:normalized_spectra}. As shown in \cref{fig:metric_error_forbild}{\color{red}b}, metrics $\text{RE}_{\bx}^{n}$ converge to constants that are determined largely by the levels of data noise. The basis sinograms obtained at iteration 10$^2$ from the noisy data are shown in \cref{fig:sinogram_forbild_noisy}. In \cref{fig:recon_forbild_noisy}, we display basis images and VMIs obtained at 60 keV and 100 keV from the noisy data. These results reveal that basis sinograms can stably be estimated numerically.

\begin{figure}[htbp]
    \centering
    \includegraphics[width=0.68\textwidth]{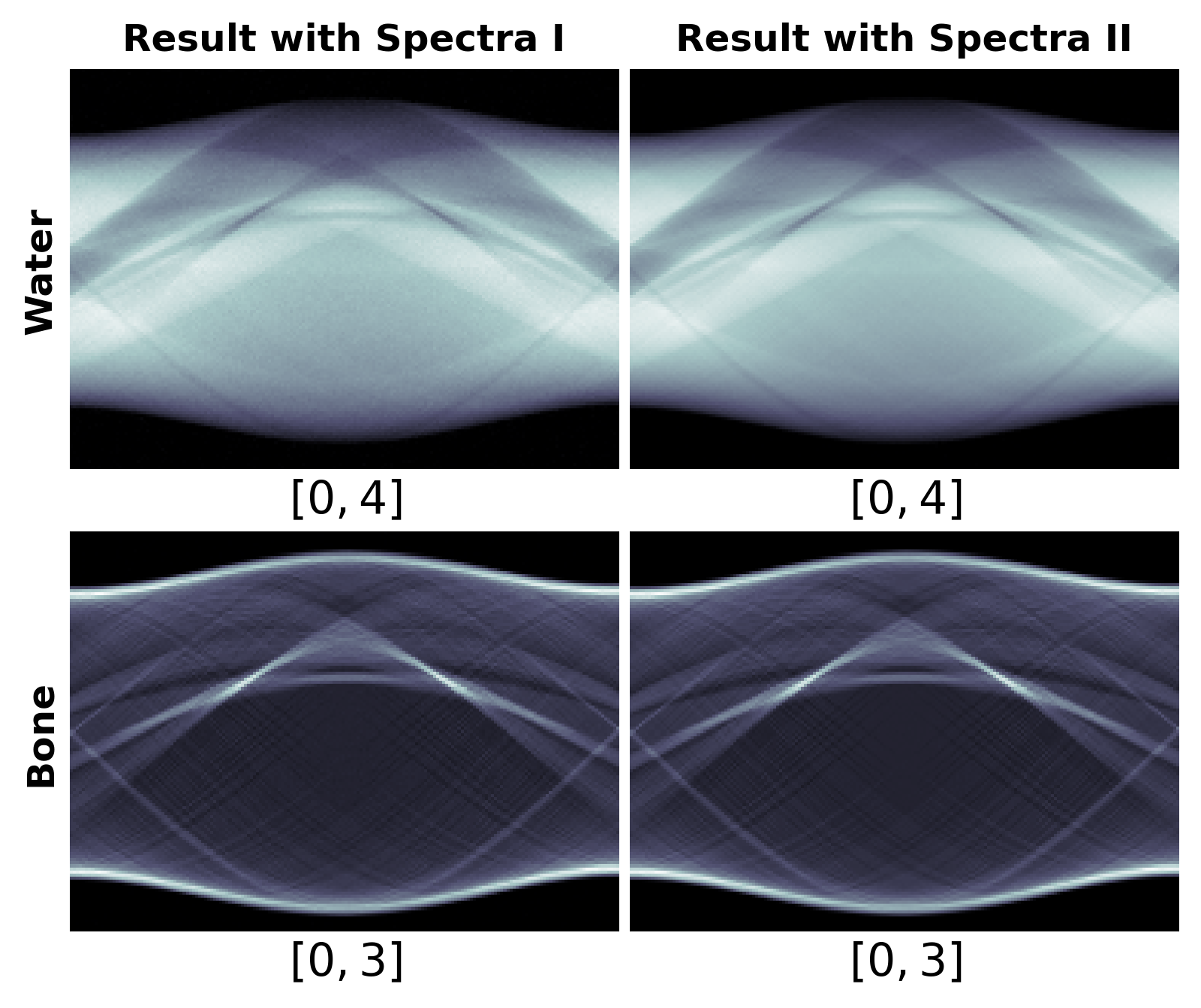}
    \caption{Estimated basis sinograms of water (row 1) and bone (row 2) of the Forbild head phantom from noisy data. Columns 1 and 2 are obtained with spectral pairs in \cref{fig:normalized_spectra}{\color{red}a} and \cref{fig:normalized_spectra}{\color{red}b}, respectively. The corresponding truth basis sinograms are displayed in column 1 of \cref{fig:sinogram_forbild}.} 
    \label{fig:sinogram_forbild_noisy}
\end{figure}

\begin{figure}[htbp]
    \centering
    \includegraphics[width=0.85\textwidth]{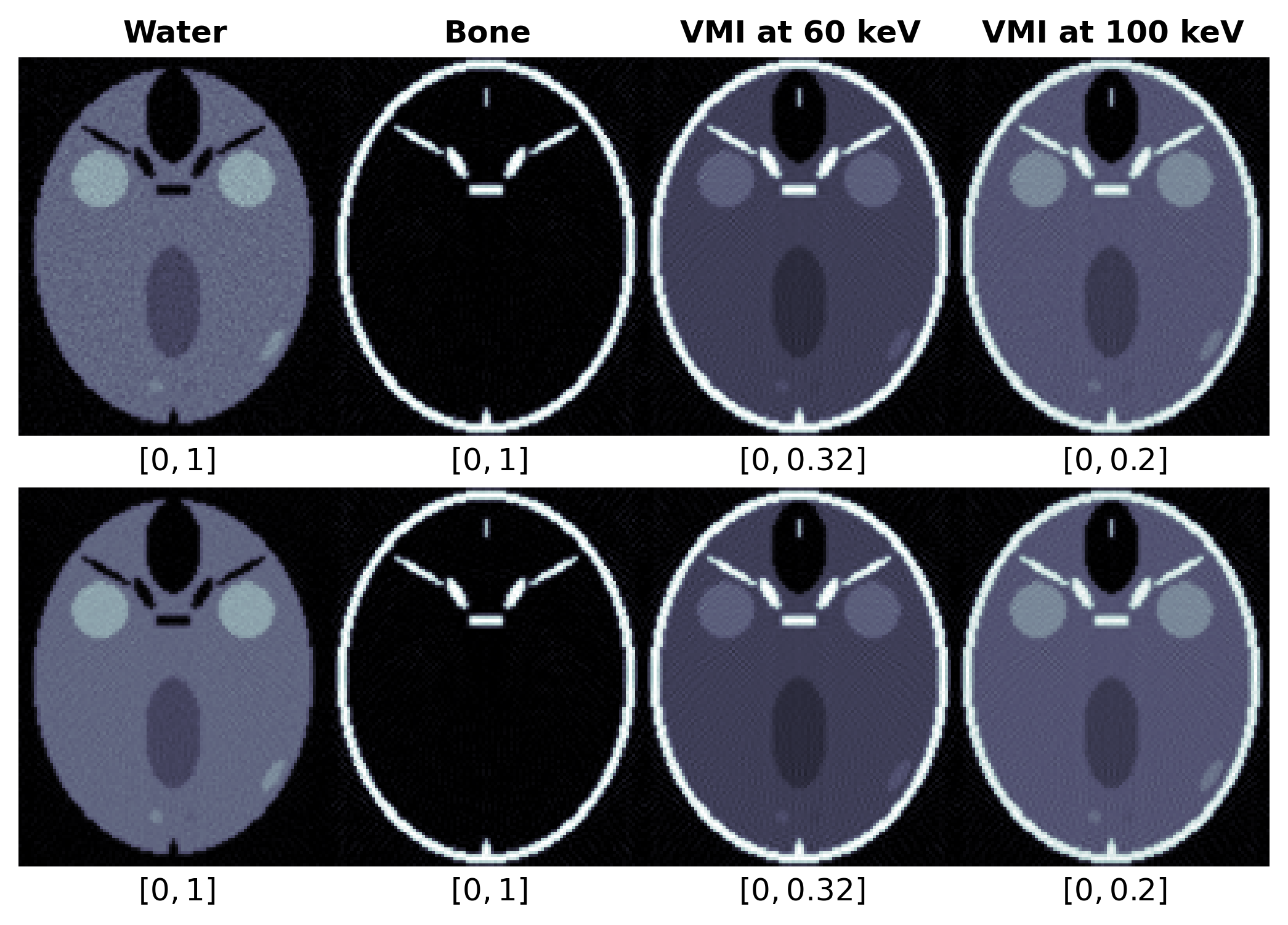}
    \caption{Basis images of water (column 1) and bone (column 2) and VMIs at energies 60 keV (column 3) and 100 keV (column 4) reconstructed from noisy data, respectively, with the spectral pairs in \cref{fig:normalized_spectra}{\color{red}a}   (row 1) and \cref{fig:normalized_spectra}{\color{red}b} (row 2) of the Forbild head phantom, respectively.}
    \label{fig:recon_forbild_noisy}
\end{figure}

\subsection{Numerical study with the patient-torso phantom}\label{subsec:test2}

We also perform a numerical study with a digital torso phantom, which  possesses realistic human torso anatomies because it was created from CT images of a patient. As shown in row 2 of  \cref{fig:truth_phantoms}, its truth basis images of water and bone are presented on an array of $256\times256$ identical pixels of square shapes covering an area $[-5, 5]\times [-5, 5]$ cm$^2$. Also, using the truth basis images in \cref{eq:attenuationco-DD}, we can obtain the truth VMIs  at energies 60 keV and 100 keV, which are shown also in row 2 of \cref{fig:truth_phantoms}. We assume a geometrically-consistent scan configuration with 360 parallel-ray projections uniformly sampled on $[-7.05, 7.05]$ cm at each of 360 views uniformly distributed over 180$^\circ$, and the truth basis sinograms in column 1 of \cref{fig:sinogram_torso}. Furthermore, with each pair of the spectra shown in \cref{fig:normalized_spectra}, we generate a set of the noiseless data by plugging the scan geometric parameters, the low- and high-kV spectra, MACs obtained from the NIST database, and truth basis images in row 2 of \cref{fig:truth_phantoms} into  \cref{eq:dd_formula} and \cref{eq:linear_part}.

Using the Newton method in \cref{eq:ordinary_newton}, we solve the DD-data model in \cref{eq:nonlinear_systems_map} with each of the two sets of noiseless data generated for obtaining basis sinograms.  Metrics $\text{RE}_{\bx}^{n}$ are first displayed as functions of iteration number in \cref{fig:metric_error_torso}{\color{red}a} for the two noiseless data sets. It can be observed that metrics numerically converge to the level of double-floating precision of the computer used. The basis sinograms obtained at iteration 10$^2$ are shown  in columns 2 and 3 of \cref{fig:sinogram_torso}.

The numerical results of convergence metric $\text{RE}_{\bx}^{n}$ and basis sinograms estimated reveal that under the conditions discussed, the DD-data model in \cref{eq:nonlinear_systems_map}, can accurately be solved by use of the ordinary Newton method.  Furthermore, it can be observed in \cref{fig:metric_error_torso} that the pair of spectra in \cref{fig:normalized_spectra}{\color{red}b} with a degree of overlapping lower than that of the other pair of spectra in \cref{fig:normalized_spectra}{\color{red}a} may be beneficial to solving the nonlinear system \cref{eq:nonlinear_systems_map}.

We also apply the FBP algorithm to reconstructing the basis images from the corresponding basis sinograms estimated. In \cref{fig:recon_torso}, we display the basis images reconstructed for both pairs of spectra, and also the VMIs obtained at energies 60 keV and 100 keV with \cref{eq:attenuationco}.

To study the numerical stability, two data sets containing Gaussian noise with SNR levels of 24.7 dB and 24.3 dB were generated for the two pairs of spectra in \cref{fig:normalized_spectra}, respectively. As shown in \cref{fig:metric_error_torso}{\color{red}b}, metrics $\text{RE}_{\bx}^{n}$ converge to constants that are determined largely by the levels of data noise. The basis sinograms obtained at iteration 10$^2$ from the noisy data are shown in \cref{fig:sinogram_torso_noisy}. In \cref{fig:recon_torso_noisy}, we display basis images and VMIs obtained at 60 keV and 100 keV from the noisy data. These results reveal that basis sinograms can stably be estimated numerically.

\begin{figure}[htbp]
    \centering
    \includegraphics[width=1.\textwidth]{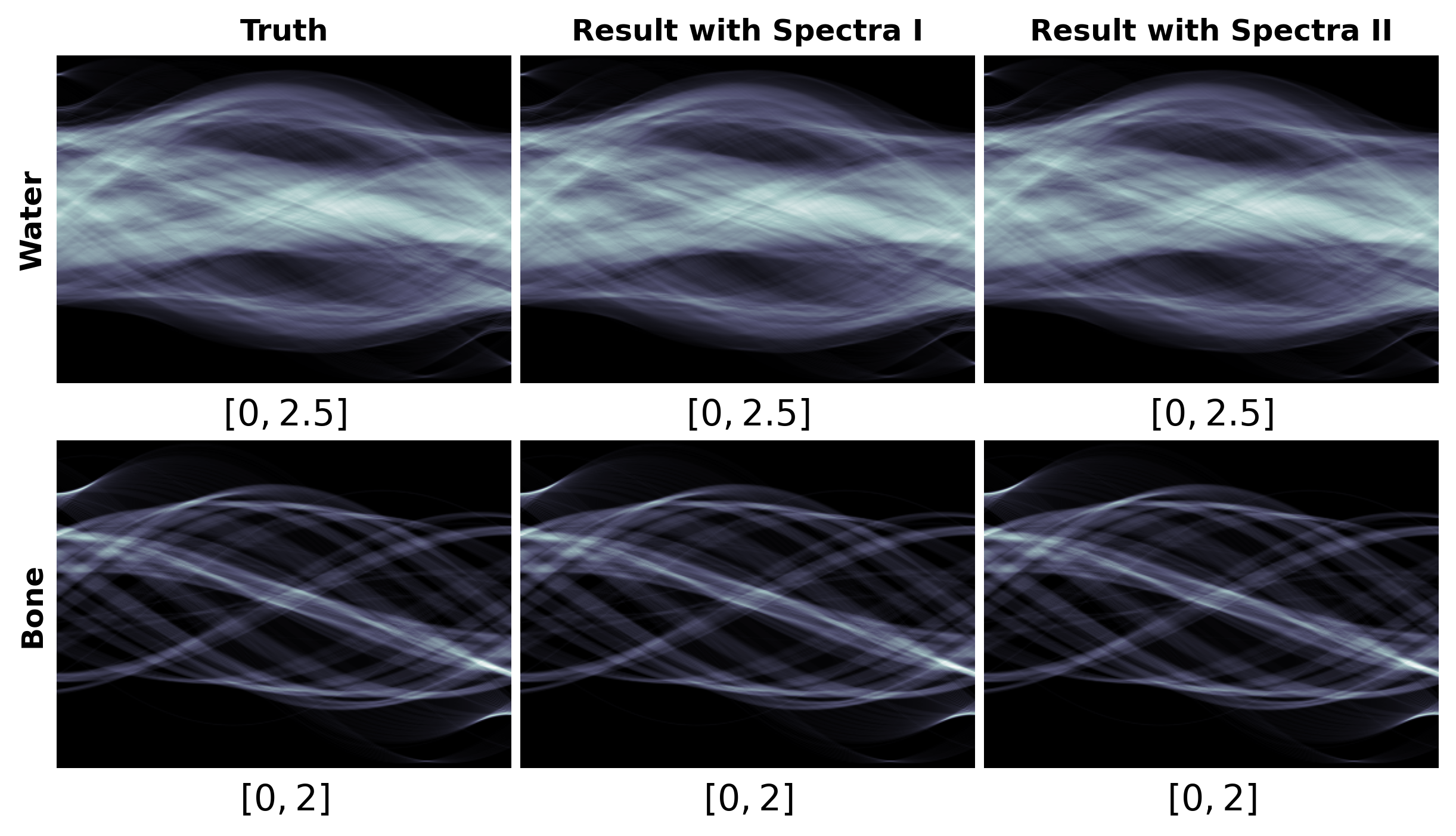}
      \caption{Truth (column 1) and estimated (columns 2 and 3) basis sinograms of water (row 1) and bone (row 2) of the torso phantom from noiseless data. Columns 2 and 3 are obtained with spectral pairs in \cref{fig:normalized_spectra}{\color{red}a} and \cref{fig:normalized_spectra}{\color{red}b}, respectively.} 
    \label{fig:sinogram_torso}
\end{figure}

\begin{figure}[htbp]
\centering
{\includegraphics[width=.49\textwidth]{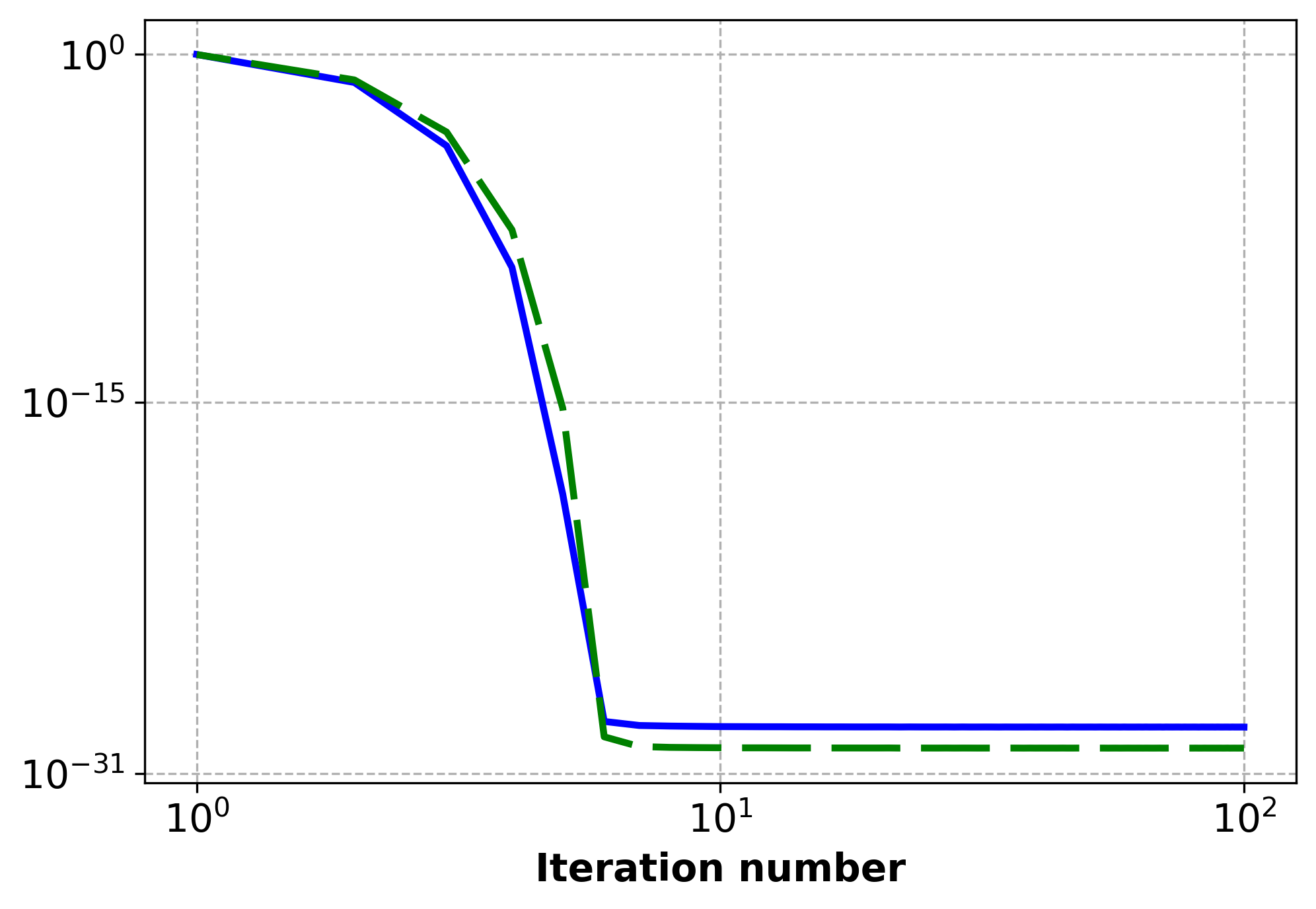}}
{\includegraphics[width=.49\textwidth]{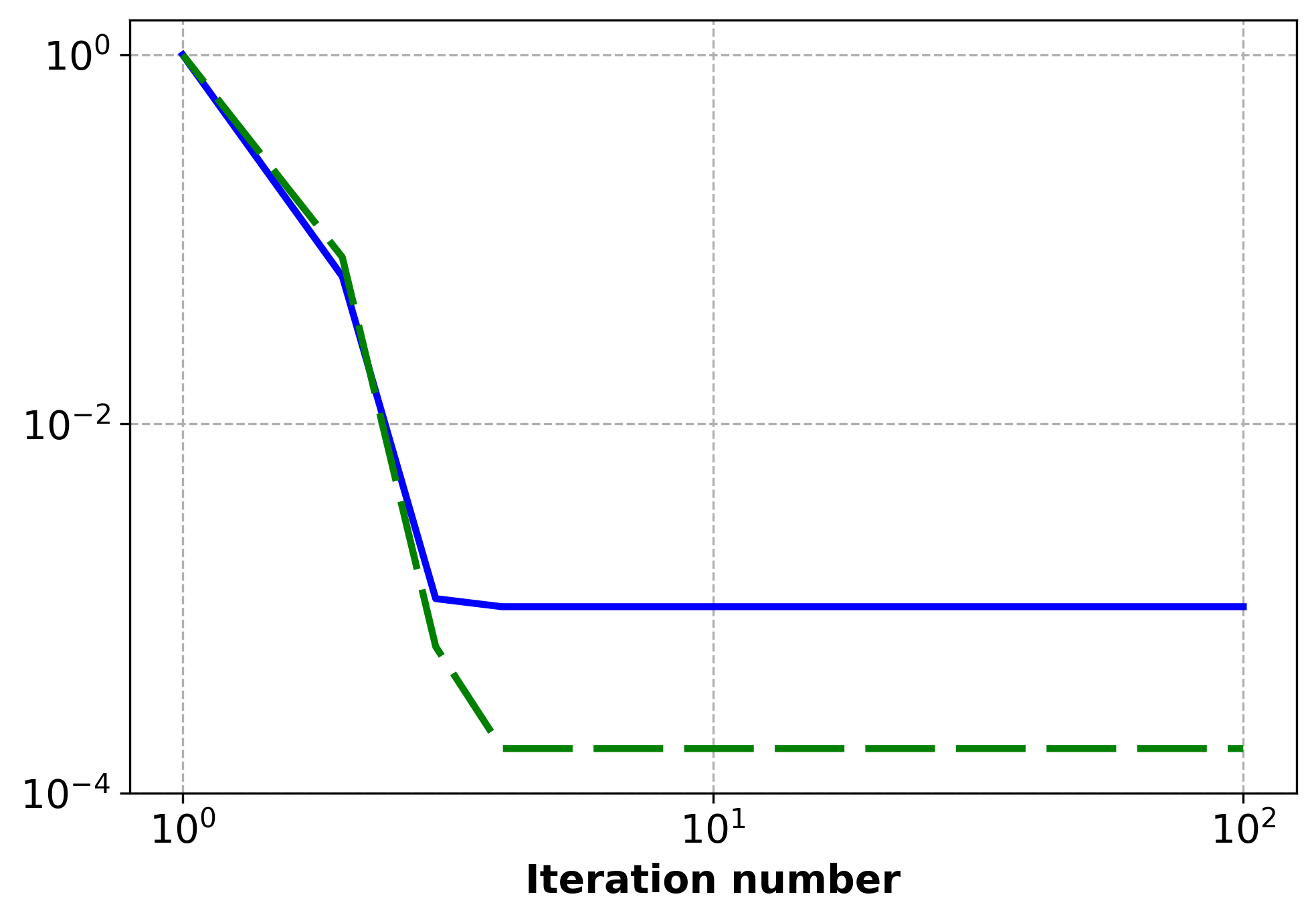}}\\
\hspace{1.0cm} (a) \hspace{5.5cm} (b)
\caption{Metrics $\text{RE}_{\bx}^{n}$ plotted in a log-log scale as functions of the iteration number of the ordinary Newton method obtained with the spectral pairs in \cref{fig:normalized_spectra}{\color{red}a} (solid) and \cref{fig:normalized_spectra}{\color{red}b} (dashed), respectively, from noiseless data (a) and noisy data (b) of the torso phantom. 
}
\label{fig:metric_error_torso}
\end{figure}

\begin{figure}[htbp]
    \centering
    \includegraphics[width=1.\textwidth]{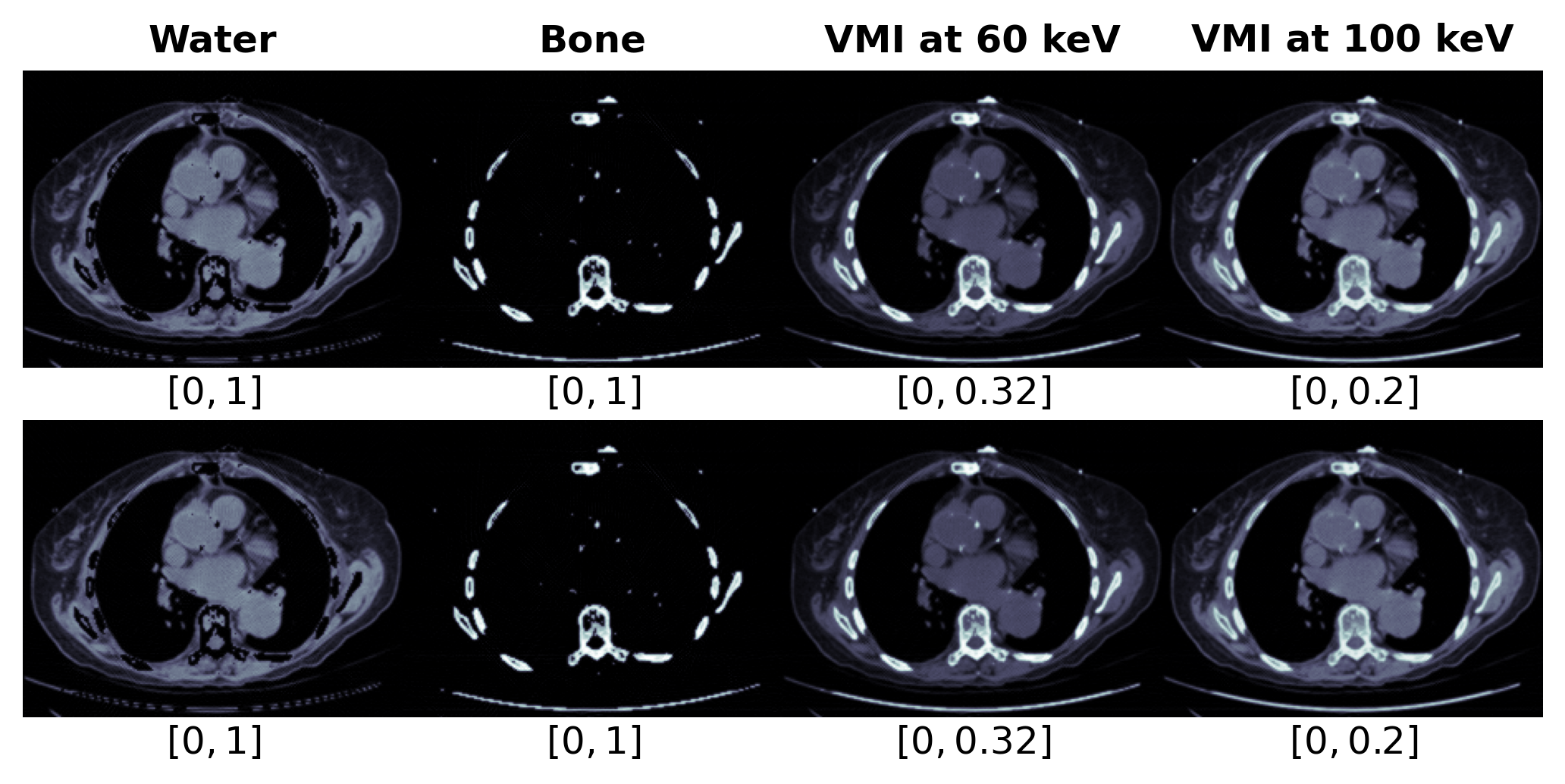}
     \caption{Basis images of water (column 1) and bone (column 2) and VMIs at energies 60 keV (column 3) and 100 keV (column 4) reconstructed from noiseless data, respectively, with the spectral pairs in \cref{fig:normalized_spectra}{\color{red}a}   (row 1) and \cref{fig:normalized_spectra}{\color{red}b} (row 2) of the torso phantom, respectively.} 
    \label{fig:recon_torso}
\end{figure} 

\begin{figure}[htbp]
    \centering
    \includegraphics[width=0.7\textwidth]{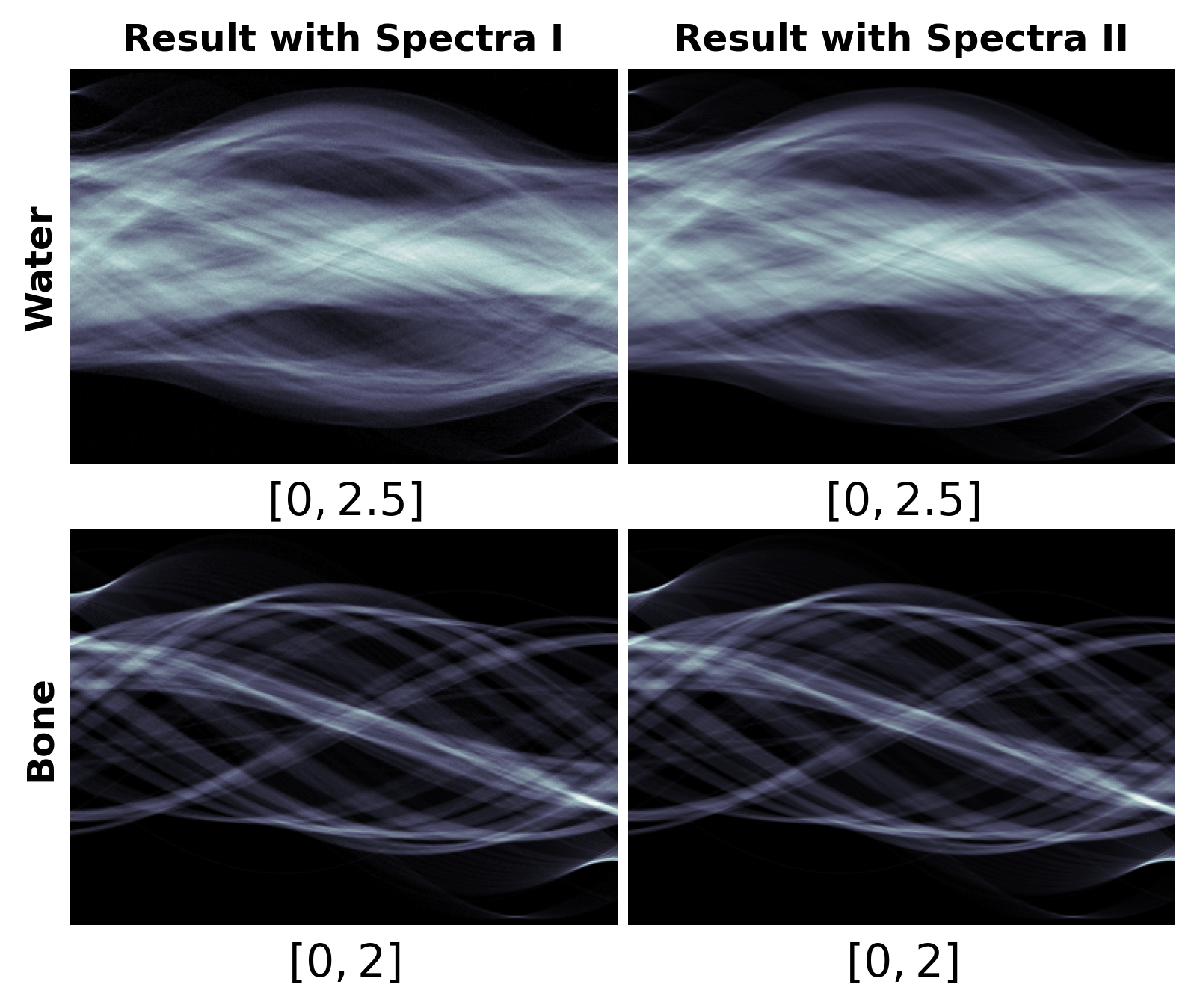}
      \caption{Estimated basis sinograms of water (row 1) and bone (row 2) of the torso phantom from noisy data. Columns 1 and 2 are obtained with spectral pairs in \cref{fig:normalized_spectra}{\color{red}a} and \cref{fig:normalized_spectra}{\color{red}b}, respectively. The corresponding truth basis sinograms are displayed in column 1 of \cref{fig:sinogram_torso}.} 
    \label{fig:sinogram_torso_noisy}
\end{figure}

\begin{figure}[htbp]
    \centering
    \includegraphics[width=1.\textwidth]{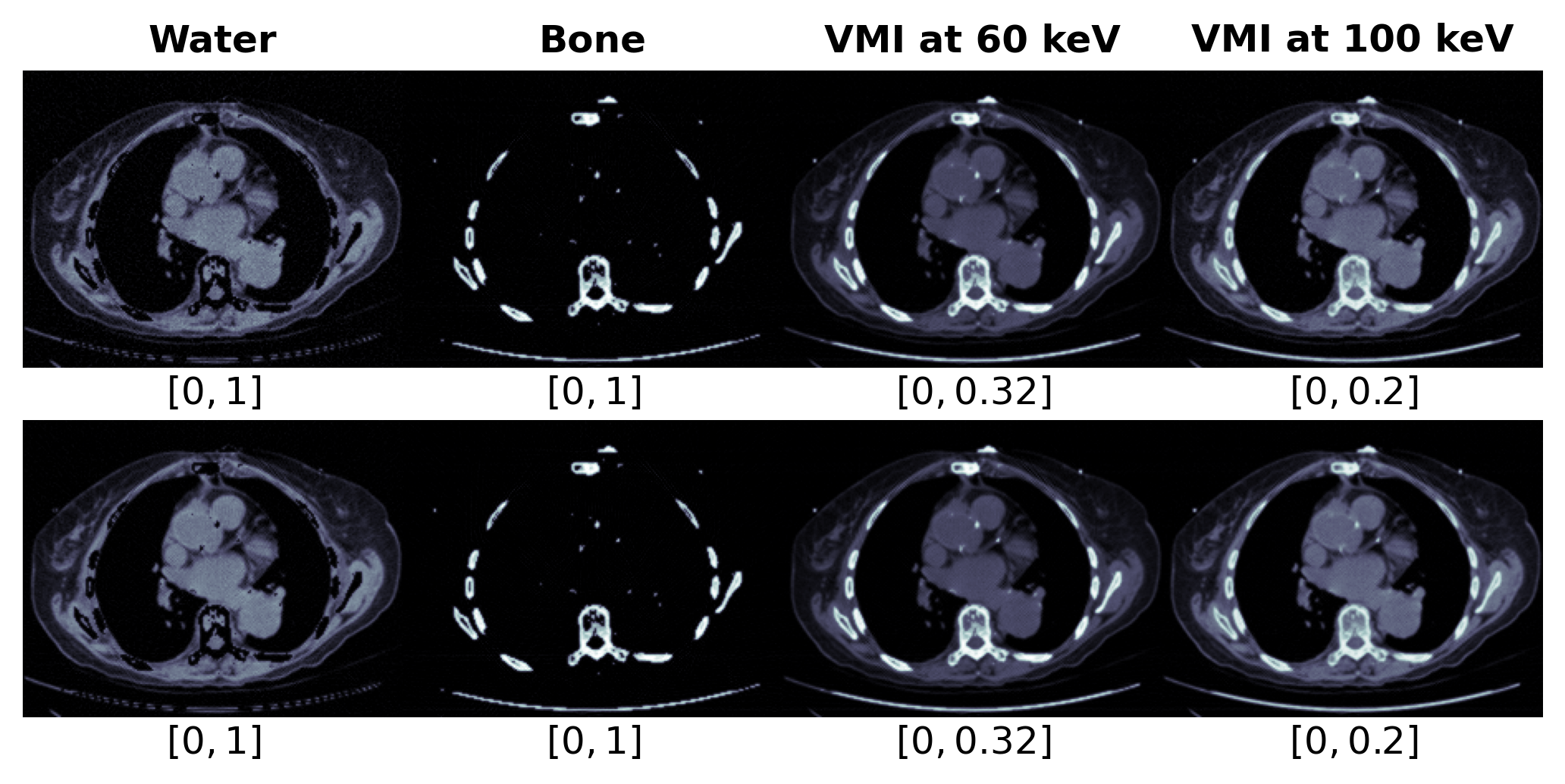}
     \caption{Basis images of water (column 1) and bone (column 2) and VMIs at energies 60 keV (column 3) and 100 keV (column 4) reconstructed from noisy data, respectively, with the spectral pairs in \cref{fig:normalized_spectra}{\color{red}a}  (row 1) and \cref{fig:normalized_spectra}{\color{red}b} (row 2) of the torso phantom, respectively.} 
    \label{fig:recon_torso_noisy}
\end{figure}

\section{Discussion and Conclusion}\label{sec:discussion_conclusion}

The two-step DDD method is used widely for reconstruction of quantitative images in geometrically-consistent MSCT.  We investigated the existence, uniqueness, and stability of the solution to DD-data model \cref{eq:nonlinear_systems_original} or \cref{eq:nonlinear_systems_map} in geometrically-consistent MSCT. We have derived a sufficient condition that the nonlinear mapping in \eqref{eq:nonlinear_systems_map} 
is a local homeomorphism, and also a necessary condition that it is a proper mapping and then further a homeomorphism, where the homeomorphism implies the existence of the solution. 
In particular, for the DECT case, we demonstrate that the corresponding mapping is a proper mapping, which implies a sufficient condition on the solution existence. 
We also derived a sufficient condition on the global injectivity of the nonlinear mapping in \eqref{eq:nonlinear_systems_map}, which is equivalent to the uniqueness condition of the solution. 
Additionally, we identified some bounded regions for a specific stability estimate of the discrete model.

Furthermore, we conducted quantitative studies to demonstrate numerically the validity extent of proposed conditions. The results from the ordinary Newton method with noise-free, ideal data suggest that the truth basis sinograms can numerically accurately be recovered. We also demonstrate numerically the solution stability by using noisy data.

To the best of our knowledge, this is the first work that investigates the specific and explicit conditions for the existence, uniqueness, and stability of the discrete basis sinogram to the DD-data model \cref{eq:nonlinear_systems_original} in 
geometrically-consistent MSCT. The conditions discussed depend on the distributions of the used energy spectra  and expansion coefficients of basis images. 
With spectra and basis material attenuation coefficients of practical relevance in diagnositic DECT, one can readily validate the existence, uniqueness, and stability of the solution to the DD-data model in advance without knowing a specific solution.

As we can see from \cref{sec:dd_model}, DD-data model \cref{eq:nonlinear_systems_original} or \cref{eq:nonlinear_systems_map} is independent of the specific discretization schemes 
of the image and data spaces. Hence, the solution analysis in the work can readily be applied to other discretization forms. 
These characteristics may imply that the proposed conditions have a broad application prospect in 
the general physical case for practical MSCT. The theoretical and numerical studies of solution may also provide insights into the design of one-step algorithms for solving directly the DD-data model.

\clearpage                                        
\appendix
\section{Supplementary data}
\label{appendix:data}

We present specific spectra pairs and MACs used in numerical studies, which are referred to as $S$ and $B$ defined in \cref{eq:jacobian_matrix}. The first and second columns of $S_1^{\tra}$ are low-kV and high-kV spectra shown in  \cref{fig:normalized_spectra}{\color{red}a}, whereas the first and second columns of $S_2^{\tra}$ are low-kV and high-kV spectra shown in  \cref{fig:normalized_spectra}{\color{red}b}. 
\begin{equation*}
    S_{1}^{\tra} = \begin{bmatrix}
    6.07397e-09&1.33388e-09\\
    7.10972e-02&2.69644e-02\\
    2.75239e-01&1.65195e-01\\
    2.75638e-01&1.98221e-01\\
    1.99787e-01&1.69333e-01\\
    1.23729e-01&1.56827e-01\\
    5.45098e-02&7.83202e-02\\
    0.00000e+00&6.44024e-02\\
    0.00000e+00&5.06969e-02\\
    0.00000e+00&3.78745e-02\\
    0.00000e+00&2.66297e-02\\
    0.00000e+00&1.69125e-02\\
    0.00000e+00&8.18109e-03\\
    0.00000e+00&4.43064e-04
    \end{bmatrix},
    S_{2}^{\tra} = \begin{bmatrix}
    6.07397e-09&0.00000e+00\\
    7.10972e-02&0.00000e+00\\
    2.75239e-01&6.63807e-05\\
    2.75638e-01&1.42761e-02\\
    1.99787e-01&8.19111e-02\\
    1.23729e-01&1.77453e-01\\
    5.45098e-02&1.38680e-01\\
    0.00000e+00&1.49775e-01\\
    0.00000e+00&1.40374e-01\\
    0.00000e+00&1.17542e-01\\
    0.00000e+00&8.91056e-02\\
    0.00000e+00&5.94650e-02\\
    0.00000e+00&2.97105e-02\\
    0.00000e+00&1.64171e-03
    \end{bmatrix}.
\end{equation*}
The first and second columns of matrix $B^{\tra}$ are the MACs of water and bone materials, respectively. 
    \begin{equation*}
    B^{\tra} = \begin{bmatrix}
    4.76251e+00&2.55327e+01\\
    7.75665e-01&3.78214e+00\\
    3.64996e-01&1.26821e+00\\
    2.65875e-01&6.50013e-01\\
    2.25389e-01&4.16057e-01\\
    2.05162e-01&3.11231e-01\\
    1.92592e-01&2.55515e-01\\
    1.83292e-01&2.21535e-01\\
    1.76097e-01&1.99266e-01\\
    1.70448e-01&1.84934e-01\\
    1.65911e-01&1.76111e-01\\
    1.62055e-01&1.70366e-01\\
    1.58449e-01&1.65270e-01\\
    1.55060e-01&1.59231e-01
    \end{bmatrix}.
    \end{equation*}
\clearpage    
\bibliographystyle{plain}
\bibliography{MCTreferences}

\end{document}